\documentclass[12pt]{amsart}
\usepackage{amsfonts}
\usepackage{amscd}
\usepackage{amssymb}
\usepackage{graphics}

\usepackage{amsmath}
\usepackage{graphicx,psfrag}
\usepackage{subcaption,mathtools}

\usepackage{hyperref}
\hypersetup{colorlinks,citecolor=blue}
\newcommand{\la}{\langle}
\newcommand{\ra}{\rangle}
\newtheorem{theorem}{Theorem}

\newtheorem{corollary}[theorem]{Corollary}

\newtheorem{definition}[theorem]{Definition}

\newtheorem{lemma}[theorem]{Lemma}

\newtheorem{proposition}[theorem]{Proposition}
\newtheorem{remark}[theorem]{Remark}

\newcommand{\BZ}{{\mathbb{Z}}}

\newcommand{\Cl}{\mathrm{Cl}}

\newcommand{\zz}{\mathbb{Z}[\frac{1}{2}]}

\newcommand{\Ztwo}{\mathbb{Z}^2}

\begin{document}
	\author{Gili Golan Polak}\thanks{This research was supported by the ISRAEL SCIENCE FOUNDATION  (grant 2322/19).}
	%\address{Ben Gurion University of the Negev}
	%\email{golangi@bgu.ac.il}

	\title{The ``spread'' of Thompson's group $F$}
	
%	
%	\begin{abstract}
%		Recall that a group $G$ is said to be \textit{$\frac{3}{2}$-generated} if  every non-trivial element $g\in G$ has a co-generator in $G$ (i.e., an element which together with $g$ generates $G$). 
%		Donoven and Harper proved that Thompson's group $V$ is .... 
%		Recently, we proved that Thompson's group $F$ is ``almost'' $\frac{3}{2}$-generated in the sense that every element of $G$ whose image in the ab
%	\end{abstract}

	\begin{abstract}
%		Recall that a group $G$ is said to be $\frac{3}{2}$-generated if every non-trivial element of $G$ is part of a generating pair of $G$. The \emph{spread} of a group $G$ is the supremum over all integers $k$ such that for every non-trivial $k$ elements 
%		
%		**
%		
%		Let $S$ be a subset of a group $G$. We say that an element $g\in G$ is a \emph{co-generator} of the elements of 
		Recall that a group $G$ is said to be \textit{$\frac{3}{2}$-generated} if  every non-trivial element $g\in G$ has a co-generator in $G$ (i.e., an element which together with $g$ generates $G$). 
		Thompson's group $V$ was proved to be $\frac{3}{2}$-generated by Donoven and Harper in 2019. It was the first example of an infinite finitely presented non-cyclic $\frac{3}{2}$-generated group. 
		In 2022, Bleak, Harper and Skipper proved that Thompson's group $T$ is also $\frac{3}{2}$-generated.
		Since the abelianization of Thompson's group $F$ is $\Ztwo$, it cannot be $\frac{3}{2}$-generated. However, we 	recently proved that Thompson's group $F$ is ``almost'' $\frac{3}{2}$-generated in the sense that every element of $F$ whose image in the abelianization forms part of a generating pair of $\mathbb{Z}^2$ is part of a generating pair of $F$.
		
		A natural generalization of $\frac{3}{2}$-generation is the notion of spread. Recall that the \emph{spread} of  a group $G$ is the supremum over all integers $k$ such that every $k$ non-trivial elements of $G$ have a common co-generator in $G$. The \emph{uniform spread} of a group $G$ is the supremum over all integers $k$ for which there exists a conjugacy class $C\subseteq G$ such that every $k$ non-trivial elements of $G$ have a common co-generator which belongs to $C$. In this paper we study
		 modified versions of these notions for Thompson's group $F$.		
	\end{abstract}
	
	\maketitle

	\section{Introduction}

	A group $G$ is said to be $\frac{3}{2}$-generated if every non-trivial element of $G$ is part of a generating pair of $G$. In 2000, settling a problem of Steinberg from 1962 \cite{S}, Guralnick and Kantor proved that all finite simple groups are $\frac{3}{2}$-generated \cite{GK}. 
	In 2008, Breuer, Guralnick and Kantor \cite{BGK} observed that if a group $G$ is $\frac{3}{2}$-generated then every proper quotient of it must be cyclic. They conjectured that for finite groups this is also a sufficient condition. The conjecture was proved in 2021 by Burness, Guralnick and Harper \cite{BGH}.
		In fact, Burness, Guralnick and Harper proved a much stronger theorem. % regarding the spread of finite groups. 
	%which solved a problem posed by Brenner and Wiegold in 1975 \cite{}. 
%	Recall that the spread of a group  $G$ is defined as follows. 
%		\begin{definition}
%		Let $G$ be a group. The \emph{spread} of $G$, denoted $s(G)$ is the supremum over all integers $k$ such that for every $k$ non-trivial elements $h_1,\dots,h_k\in G$ there exists an element $g\in G$, such that $\la h_k,g\ra =G$.
%	\end{definition}
	Recall that the \emph{spread} of a group $G$, denoted $s(G)$, is the supremum over all integers $k$ such that any $k$ non-trivial elements of $G$ have a common co-generator in $G$.
	  The notion of spread was first introduced in 1975 by  Brenner and Wiegold who studied the spread of soluble groups and of some families of simple groups \cite{BrWi}.  Since then the spread of finite groups has attracted a lot of attention (see, for example, \cite{BrWi,BrWi2,BGK,BH,BGH,GK,GSha,H}).
	  
	  % was studied by Brenner, Breuer, Burness, Guralnick, Wiegold, Kantor and Shalev among others \cite{}. 
	  By definition, a group $G$ is $\frac{3}{2}$-generated if and only if the spread $s(G)\geq 1$. In \cite{BGH}, Burness, Guralnick and Harper proved that if $G$ is a finite group such that every proper quotient of it is cyclic then $s(G)\geq 2$. It follows that a finite group $G$ is $\frac{3}{2}$-generated if and only if $s(G)\geq 2$.  In other words, there are no finite groups with spread $1$. 
	
	A related notion is that of uniform spread. Recall that the \emph{uniform spread} of a group $G$, denoted $u(G)$, is the supremum over all integers $k$ for which there is an element $g\in G$ %a conjugacy class $C_k\subseteq G$
	 such that every $k$ non-trivial elements of $G$ have a common co-generator that is a conjugate of $g$. %which belongs to $C_k$.
	  The term uniform spread was introduced in 2008 by Breuer, Guralnick and Kantor \cite{BGK}.
	The idea is that if $u(G)\geq m$, then the common co-generators of different $m$-tuples of non-trivial elements of $G$ can be chosen 	uniformly in a sense (from the same conjugacy class). Note that already in 1970 Binder studied the uniform spread of symmetric groups \cite{B}.
	 By definition, for every group $G$, we have $u(G)\leq s(G)$. Guralnick and Kantor \cite{GK} proved that $u(G) \geq  1$ for every finite simple group $G$. Burness, Guralnick and Kantor characterized the finite groups with $u(G)=0$ and with $u(G)=1$ \cite{BGH}.  Uniform spread of finite groups was studied also in  \cite{BH,BGu,GSha}, among others. 
	 
	 The study of $\frac{3}{2}$-generation, the spread and the uniform spread of infinite groups started only recently. 	Note that for an infinite group $G$ every proper quotient being cyclic is not a sufficient condition for the group  being $\frac{3}{2}$-generated. Indeed, the infinite alternating group $A_{\infty}$ is simple but not finitely generated. Moreover, there are finitely generated infinite simple groups which are not $2$-generated (see \cite{Guba})  and in particular, are not $\frac{3}{2}$-generated. Recently, Cox \cite{Cox} constructed an example of an infinite $2$-generated group $G$ such that every proper quotient of $G$ is cyclic and yet $G$ is not $\frac{3}{2}$-generated.

	 Obvious examples of infinite $\frac{3}{2}$-generated groups are $\mathbb{Z}$ and the Tarski monsters $T$ constructed by Olshanskii \cite{O}.  In fact, $s(\mathbb{Z})=u(\mathbb{Z})=s(T)=u(T)=\infty$. 
	 % Recall that Tarski monsters are infinite finitely generated non-cyclic groups where every proper subgroup is cyclic\footnote{There are two types of Tarski monsters. One where every proper subgroup is infinite cyclic and one where every proper subgroup is cyclic of order $p$ for some fixed prime $p$.}. %Both $\mathbb{Z}$ and the Tarski numbers $T$ satisfy $u(\mathbb{Z})=u(T)=\infty$. 
	% In particular, if $T$ is a Tarski monster, then $T$ is generated by any pair of non-commuting elements of $T$. %Both $\mathbb{Z}$ and the Tarski monsters $T$ have infinite uniform spread. 
	 % Since the center of $T$ is trivial, every non-trivial element of $T$ is part of a generating pair of $T$. In fact, the spread and the uniform spread of $T$ satisfy $s(T)=u(T)=\infty$ \cite{Cox}.
	 %
	 In 2019, Donoven and Harper gave the first examples of infinite non-cyclic $\frac{3}{2}$-generated groups, other than Tarski monsters. Indeed, they proved that Thompson's group $V$ is $\frac{3}{2}$-generated. 
	 More generally, they proved that all  Higman–Thompson groups $V_n$ (see \cite{Higman}) and all Brin–Thompson groups $nV$ (see \cite{Brin}) are $\frac{3}{2}$-generated.
	 In 2020, Cox constructed two more examples of infinite $\frac{3}{2}$-generated groups. The groups constructed by Cox have finite spread greater than $2$.  Note that it is still unknown if there is an infinite group with spread exactly one. Similarly, other than the Tarski-monsters constructed by Olshanskii \cite{O}, there are no known examples of  infinite non-cyclic groups with infinite spread.
	 In 2022, Bleak, Harper and Skipper proved that Thompson's group $T$ is $\frac{3}{2}$-generated and that its uniform spread satisfies $u(T)\geq 1$. %They asked whether the spread of $T$ is infinite and what is the precise uniform spread of $T$. 
	 
	 In this paper, we study Thompson's group $F$. Recall that Thompson's group $F$ is the group of all piecewise-linear homeomorphisms  of the interval $[0,1]$ with finitely many breakpoints where all breakpoints are dyadic fractions (i.e., numbers from $\mathbb{Z}[\frac{1}{2}]\cap (0,1)$) and all slopes are integer powers of $2$.
	 
	  Thompson's group $F$ is $2$-generated and %The derived subgroup of $F$ is infinite and simple and can be characterized as the subgroup of $F$ of all functions $f$ with slope $1$ both at $0^+$ and at $1^-$ (see \cite{CFP}). 
	 the abelianization $F/[F,F]$ is isomorphic to $\mathbb{Z}^2$. The standard abelianization map  $\pi\colon F\to \mathbb{Z}^2$ maps every function $f\in F$ to $$\pi(f)=(\log_2(f'(0^+)),\log_2(f'(1^-)))$$ (see \cite{CFP}). 
	 Since the abelianization of $F$ is $\mathbb{Z}^2$, Thompson's group $F$ cannot be $\frac{3}{2}$-generated. %(and in particular, $s(G)=u(G)=0$).
	  However, we recently proved that $F$ is ``almost'' $\frac{3}{2}$-generated in the sense that every element of $F$ whose image in the ablianization $\Ztwo$ is part of a generating pair of $\Ztwo$ has a co-generator in $F$. Hence, it is natural to study modified versions of the spread and uniform spread for Thompson's group $F$. 
	  
	  Note that for $k$ elements $x_1,\dots,x_k\in F$ to have a common co-generator in $F$, the images $\pi(x_1),\dots,\pi(x_k)$ must have a common co-generator in the abelianization $\Ztwo$. Moreover, for the elements $x_1,\dots,k_k$ to have a common co-generator in $F$ that is a conjugate of some element $g$, the element $\pi(g)$ must be a common co-generator of  $\pi(x_1),\dots,\pi(x_k)$ (and in particular, $\pi(g)$ must be part of a generating pair of $\Ztwo$).
	  
	  \begin{definition}\label{def}
	  	Let $G$ be a $2$-generated group and let $A=G/[G,G]$. Let  $\pi_G\colon G\to A$ be the abelianization map. We make the following definitions. 
	\begin{enumerate}
  		\item The \emph{semi-spread} of $G$ is the supremum over all integers $k$ such that every $k$ elements $x_1,\dots,x_k\in G$ such that $\pi_G(x_1),\dots,\pi_G(x_k)$ have a common co-generator in $A$, have a common co-generator in $G$. 
	  		\item The \emph{uniform semi-spread} of $G$ is the supremum over all integers $k$ such that the following holds. 
	  		For every $a\in A$ that is part of a generating pair of $A$ there exists $g\in G$ which satisfies the following conditions. 
	  		\begin{enumerate}
	  			\item[(i)] $\pi_G(g)=a$. 
	  			\item[(ii)] Every $k$ elements $x_1,\dots,x_k\in G$ such that $\pi_G(g)$ is a common co-generator of $\pi_G(x_1),\dots,\pi_G(x_k)$ have a common co-generator in $G$ that is  a conjugate of $g$.  
	  		\end{enumerate}
 	  	\end{enumerate}	  	
	  \end{definition}

	In \cite{G-32}, we proved that the semi-spread of Thompson's group $F$ is $\geq 1$. In this paper, we prove that the semi-spread of $F$ is $2$ and that the uniform semi-spread of $F$ is $1$.
	We also prove that there is essentially one obstruction preventing  the semi-spread and the uniform semi-spread of Thompson's group $F$ from  being infinite.

	 More generally,  
	  we study the problem of which $k$-tuples of non-trivial elements in $F$ have a common co-generator and under which conditions the common co-generator can be chosen from a predetermined conjugacy class.

	   %Let us begin with the following lemma.  
	  
	 % \vskip .2cm
	 % \subsection{The main results}
	  
	%  Let $S$ be a finite subset of $F$ and let $g\in F$. We are interested in conditions on $S$ and $F$ which guarantee that the elements of $S$ in $F$ have a common co-generator that is a conjugate of $g$. 
	  
	 %  Clearly, if $\pi(g)$ is not a common co-generator of the elements of $\pi(S)$ in $\Ztwo$, then the elements of $S$ cannot have a co-generator that is a conjugate of $g$. It turns out that there is only more obstruction. 
	  
	 Let $S$ be a finite subset of $F$ and let $g\in F$. Assume that  $\pi(g)$ is a common co-generator of the elements of $\pi(S)$ in $\Ztwo$.
	   % Let $c,d\in \mathbb{Z}$ be co-prime (i.e., such that $(c,d)$ is part of a generating pair of $\Ztwo$) and let $g\in F$ be such that $\pi(g)=(c,d)$. %. We let $F_{c,d}=\{h\in F \mid \pi(h)=(c,d)\}$. %Given a subset $S$ of $F$, we will write $S\subseteq_f F$
%	  Let $S$ be a finite subset of $F$ such that $(c,d)$ is a common co-generator of the elements of $\pi(S)$ in $\Ztwo$. %such that $\pi(g)=(c,d)$. 
	  We are interested in conditions under which the elements of $S$  have a common co-generator in $F$ that is a conjugate of $g$. Let us begin with the following lemma. 
	  
	    \begin{lemma}\label{lem:only obstruction}
	    Let $S$ be a finite subset of $F$ and let $g\in F$  be such that $\pi(g)$ is a common co-generator of the elements of $\pi(S)$ in $\Ztwo$. %and let $g\in F_{c,d}$. %such that $\pi(g)=(c,d)$. 
	  	Assume that the following conditions hold.
	  	\begin{enumerate}
	  		\item[(1)] The element $g$ fixes some number  $\alpha\in (0,1)$. 
	  		\item[(2)] Every number $\beta\in (0,1)$ is fixed by some element of $S$. 
	  	\end{enumerate}
	  	Then the elements of $S$ do not have a common co-generator in $F$ that is a conjugate of $g$.
	  \end{lemma}

	  \begin{proof}
	  	Let $\sigma \in F$. It suffices to prove that %and consider the element 
	  	$g^\sigma$ (i.e., the element $\sigma^{-1}g\sigma$) is not a common co-generator of the elements of $S$ in $F$. 
	  	Since $g$ fixes some   $\alpha\in (0,1)$, the element $g^\sigma$  fixes the number $\sigma(\alpha)\in (0,1)$. 	By Condition (2) there is an element  $f\in S$ which also fixes $\sigma(\alpha)$. Hence, $\la f,g^\sigma\ra\neq F$. 
	  \end{proof}
  
  The main result of this paper is that the converse of Lemma \ref{lem:only obstruction} holds.
  
   \begin{theorem}\label{thm:most general}
   	Let $S$ be a finite subset of $F$ and let $g\in F$. Assume that $\pi(g)$ is a common co-generator of the elements of $\pi(S)$ in $\Ztwo$. Then the elements of $S$ have a common co-generator in $F$ that is a conjugate of $g$ if and only if at least one of the following conditions holds.  
   	 %there exists $\sigma\in F$ such that $\pi(g)$ is a common co-generator of the elements of $S$ in $F$ if and only if at least one of the following conditions holds. 
  %	Let $c,d\in \mathbb{Z}$ be co-prime and let $g\in F$ be such that $\pi(g)=(c,d)$. Let $S\subseteq F$ be a finite subset such that $(c,d)$ is a common co-generator of the elements of $\pi(S)$ in $\Ztwo$. Then the elements of $S$ have a common co-generator in $F$ that is a conjugate of $g$ if and only if at least one of the following conditions holds. 
  	\begin{enumerate}
  		\item[(1)] The element $g$ does not fix any number $\alpha\in (0,1)$. 
  		\item[(2)] There is a number $\beta\in (0,1)$ that is not fixed by any of the elements in $S$. 
  	\end{enumerate}
  \end{theorem}

%Below, we prove the following elementary lemma. 
%
%\begin{lemma}
%	Let $S$ be a finite subset of $F$ such that $(c,d)$ is a common co-generator of the elements of $\pi(S)$ in $\Ztwo$. If $|cd|\neq 1$, then there is a number $\beta\in (0,1)$ that is not fixed by any of the elements of $S$. 
%\end{lemma}

%
%Now, let $S$ be a finite subset of $F$ and assume that the elements of $\pi(S)$ have a common co-generator $(c,d)\in \Ztwo$. By Lemma \ref{} below, if $|cd|\neq 1$,  there is a number $\beta\in (0,1)$ that is not fixed by any of the elements in $S$ (i.e., $S$ satisfies Condition (2) from Theorem \ref{thm:most general}). %Hence, in that case, for every $g\in F$ such that $\pi(g)$ is a common co-generator of the elements of $\pi(S)$ in $\Ztwo$ 
%Note also that if $cd=1$, there exists an element $g\in F$ such that $\pi(g)=(c,d)$ and such that $g$ does not fix any $\alpha\in (0,1)$ (see Remark \ref{} below). Hence, Theorem \ref{thm:most general} implies that if $cd\neq 1$, then the elements of $S$ have a common co-generator in $F$. Hence, we have the following.
%
%%Hence, Theorem \ref{} has the following corollary. 
%

Now, let $S$ be a finite subset of $F$ and assume that $(c,d)$ is a common co-generator of the elements of $\pi(S)$ in $\Ztwo$.  
 If $|cd|\neq 1$, then by Lemma \ref{lem:condition 2} below,  there is a number $\beta \in (0,1)$ that is not fixed by any of the elements in $S$ (i.e., Condition (2) of Theorem \ref{thm:most general} holds). Hence, in that case, by Theorem \ref{thm:most general}, for every $g\in F$ such that $\pi(g)$ %=(c,d)$, 
 is a common co-generator of the elements of $\pi(S)$ in $\Ztwo$,
 there exists $\sigma\in F$ such that $g^\sigma$ is a common co-generator of the elements of $S$ in $F$. %In particular, the elements of $S$ have a common co-generator in $F$.
  %Hence, it remains to consider the cases where $cd=-1$ and where $cd=1$. 
  
If $cd=-1$, then there is an element $g\in F$ such that $\pi(g)=(c,d)$ and such that $g$ does not fix any $\alpha\in (0,1)$ (see Remark \ref{rem:cd<0} below). Then, by Theorem \ref{thm:most general}, for every such element $g$, the elements of $S$ have a common co-generator in $F$ that is a conjugate of $g$. 

Finally, if $cd=1$, then by Remark \ref{rem:cd>0} below, every element $g\in F$ such that $\pi(g)=(c,d)$ fixes some number  $\alpha\in (0,1)$ (i.e., does not satisfy Condition (1) from Theorem \ref{thm:most general}). Hence, if the elements of $\pi(S)$ do not have another common co-generator $(c_1,d_1)$ such that $c_1d_1\neq 1$, then the elements of $S$ have a common co-generator in $F$ if and only if Condition (2) from Theorem \ref{thm:most general} holds for $S$. 
Thus, we have the following. 

%the elements of $S$ have a common co-generator in $F$ if and only if Condition (2) from Theorem \ref{thm:most general} holds or 

 %there is no element $g\in F$ such that $\pi(g)=(c,d)$ and such that $g$ does not fix any number in $(0,1)$ (see Remark \ref{} below). 
%Hence, Theorem \ref{thm:most general} has the following corollary.

\begin{corollary}\label{cor:spread}
	Let $S$ be a finite subset of $F$ such that the elements of $\pi(S)$ have a common co-generator in $\Ztwo$. Then the elements of $S$ do not have a common co-generator in $F$ if and only if the following assertions hold. 
	\begin{enumerate}
		\item The only co-generators of the elements of $S$ in $\Ztwo$ are $(1,1)$ and $(-1,-1)$. 
		\item Every number $\beta\in (0,1)$ is fixed by some element of $S$.
	\end{enumerate}
\end{corollary}

%\begin{proof}
%	If Condition (1) does not hold, then the elements of $\pi(S)$ have a common co-generator $(c,d)$ in $\Ztwo$ such that $cd\neq -1$. Hence, by the preceding remarks, the elements of $S$ have a common co-generator in $F$. If Condition (2) does not hold, then by Theorem \ref{thm:most general}, the elements of $S$ have a common co-generator in $F$. 
%	
%	
%\end{proof}

%\begin{corollary}
%	Let $c,d\in \mathbb{Z}$ be co-prime. Let $g\in F$ be such that $\pi(g)=(c,d)$. Let $S$ be a finite subset of $F$ such that $(c,d)$ is a common co-generator of the elements of $S$ in $F$. Then the following assertions hold. 
%	\begin{enumerate}
%		\item If $|cd|\neq 1$ then there exists an element $\sigma\in F$ such that $g^\sigma$ is a co-generator of the elements of $S$ in $F$. 
%		\item If $cd=-1$ and $g$ does not fix any number in $(0,1)$ then there exists an element 
%	\end{enumerate}
%\end{corollary}

	   \begin{remark}\label{rem:3 elements}
	Let $S$ be a subset of $F$ which satisfies Conditions (1) and (2) from Corollary \ref{cor:spread}. Then for every $f\in S$ there exists $n\in\mathbb{Z}$ such that $\pi(f)=(n,n\pm 1)$. Moreover, there exist at least three distinct elements $f_1,f_2,f_3\in S$ such that $\pi(f_1)=(0,\pm 1)$, $\pi(f_2)=(\pm 1,0)$ and such that $\pi(f_3)\notin\{(0,\pm 1),(\pm 1,0)\}$. 
\end{remark}

\begin{proof}
	Let $f\in S$ and let $\pi(f)=(n,m)$ for some $n,m\in\mathbb{Z}$. Since $S$ satisfies Condition (1) from Corollary \ref{cor:spread}, the element $(1,1)$ is a co-generator of $(n,m)$ in $\Ztwo$. Hence, $|n\cdot 1-m\cdot 1|=1$. It follows that $m=n\pm 1$, so $\pi(f)=(n,n\pm 1)$, as claimed. 	   	Hence, it suffices to prove that there are elements $f_1,f_2,f_3\in S$ as described. 
	
	First, we claim that there is an element $f_1\in F$ which has trivial slope at $0^+$. Indeed, if every element in $S$ has non-trivial slope at $0^+$, then there is a small right neighborhood $I_0$ of $0$, such that every element of $S$ acts linearly with  non-trivial slope in $I_0$ and as such does not fix any $\alpha\in I_0\setminus\{0\}$, in contradiction to Condition (2) from Corollary \ref{cor:spread}. 
	Hence, there must exist an element $f_1\in S$ such that the first coordinate of $\pi(f_1)$ is $0$. Then from the first claim in the lemma it follows that $\pi(f_1)=(0,\pm 1)$.
	Similarly, one can prove that there is an element $f_2\in S$ such that the slope of $f_2$ at $1^-$ is trivial. It follows that $\pi(f_2)=(\pm 1,0)$.   	
	Finally, if $\pi(f_1)=(0,\pm 1)$ and $\pi(f_2)=(\pm 1,0)$ are the only elements in $\pi(S)$ then $(1,-1)$ is a common co-generator of the elements of $\pi(S)$ in $\Ztwo$, in contradiction to Condition (1) from Corollary \ref{cor:spread}. Hence,
	there exists $f_3\in S$ such that $\pi(f_3)\neq (0,\pm 1),(\pm 1,0)$. 
\end{proof}

Remark \ref{rem:3 elements} shows that subsets $S\subseteq F$ which satisfy  Conditions (1) and (2) from  Corollary \ref{cor:spread} are subject to several constraints and  must contain at least $3$ distinct elements. %(and there are clear restrictions on their slopes at $0^+$ and at $1^-$).
Note that there exist subsets  
%Note that it is simple 
%to construct subsets
$S\subseteq F$ of size $3$ which satisfy the conditions from Corollary \ref{cor:spread}. Indeed, let $f_1,f_2,f_3\in F$ be elements such that $\pi(f_1)=(1,0)$, $\pi(f_2)=(0,1)$, $\pi(f_3)=(2,1)$ and such that $f_1$ fixes the interval $[\frac{1}{2},1]$ pointwise and $f_2$ fixes the interval $[0,\frac{1}{2}]$ pointwise. Then $S=\{f_1,f_2,f_3\}$ satisfies  Conditions (1) and (2) from Corollary \ref{cor:spread}. Hence, the minimal size of a subset $S\subseteq F$ which satisfies Conditions (1) and (2) from  Corollary \ref{cor:spread} is $3$.

Recall that by Corollary \ref{cor:spread}, if $S\subseteq F$ is a finite subset such that the elements of $\pi(S)$ have a common co-generator in $\Ztwo$ then the elements of $S$ have a common co-generator in $F$ if and only if 
Condition (1) or Condition (2) from the corollary does not hold.
% and such that Condition (1) or (2) from the corollary does not hold, then the elements of $S$ have a common co-generator in $F$.
Thus, Corollary  \ref{cor:spread} describes  %what we might think of as the ``spread'' 
the semi-spread of Thompson's group $F$. 
%Indeed, it is natural to define the \emph{semi-spread} of Thompson's group $F$ to be the supremum over all integers $k$ for which the following holds. If $f_1,\dots,f_k\in F$  are such that $\pi(f_1),\dots,\pi(f_k)$ have a common co-generator in $\Ztwo$, then they have a common co-generator in $F$. 
Indeed, in view of Remark \ref{rem:3 elements}, Corollary \ref{cor:spread} implies that  any two elements $f_1,f_2\in F$ such that $\pi(f_1),\pi(f_2)$ have a common co-generator in $\Ztwo$, have a common co-generator in $F$. %Remark \ref{rem:3 elements}
The above discussion also implies that 
%As it 
this is no longer true for every $3$ elements in $F$. Hence, the semi-spread of $F$ is $2$. However, if we restrict our attention to finite subsets $S\subseteq F$ such that the elements of $\pi(S)$ have a common co-generator $(c,d)$ in $\Ztwo$ such that $cd\neq 1$, then the ``semi-spread'' of $F$ can be considered to be infinite.
 % Note also that by Corollary \ref{cor:spread}, if we restrict our attention to $k$-tuples $f_1,\dots,f_k\in F$ such that $\pi(f_1),\dots,\pi(f_k)$ have a common co-generator $(c,d)\in\Ztwo$ such that $cd\neq 1$, then the 
% also shows that if we restrict our attention to finite subsets $S\subseteq F$ such that the elements of $\pi(S)$ have a common co-generator $(c,d)\notin \{(1,1),(-1,-1)\}$, then the ``spread'' of $F$ is infinite. 
%we can think of the ``spread'' of $F$ as being $2$. However,
%Note that Corollary \ref{cor:spread} also shows that %if we restrict our attention to finite subsets $S\subseteq F$ such that the elements of $\pi(S)$ have a common co-generator $(c,d)\notin \{(1,1),(-1,-1)\}$, then the ``spread'' of $F$ is infinite. 
%barring the specific case   
%%It also shows that barring the 
%%obstruction  
%discussed above,
%one can consider the ``spread'' of $F$ to be infinite. 

Recall that by Theorem \ref{thm:most general}, if $S$ is a finite subset of $F$ such that there is a number $\beta\in (0,1)$ that is not fixed by any of the elements in $S$ and $g\in F$ is such that $\pi(g)$ is a common co-generator of the elements of $\pi(S)$ in $\Ztwo$, then the elements of $S$ have a common co-generator in $F$ that is a conjugate of $g$. In particular, since every non-trivial element $f\in F$ does not fix some number in $(0,1)$, we have the following. 

%
%
%
%Theorem \ref{thm:most general} and the remarks following it also imply the following.
%
%
%
%
%\begin{corollary}\label{cor:detailed}
%	Let $c,d\in \mathbb{Z}$ be co-prime (i.e., such that $(c,d)$ is part of a generating pair of $\Ztwo$) and let $g\in F$  be such that $\pi(g)=(c,d)$. Let $S$ be a finite subset of $F$ such that $(c,d)$ is a common co-generator of the elements of $\pi(S)$ in $\Ztwo$. Then the following assertions hold. %and assume that the elements f
%	\begin{enumerate}
%		\item If $|cd|\neq 1$ then there is an element $\sigma\in F$ such that $g^\sigma$ is a common co-generator of the elements of $S$ in $F$. 
%		\item If $cd=-1$ and $g$ does not fix any $\alpha\in (0,1)$ or there is a number $\beta\in (0,1)$ that is not fixed by any of the elements of $S$, then there  is an element $\sigma\in F$ such that $g^\sigma$ is a common co-generator of the elements of $S$ in $F$. 
%		\item If $cd=1$ and there is a number $\beta$ that is not fixed by any of the elements of $S$ in $F$ %(i.e., if the elements of $S$ have some co-generator in $F$),
%		 then there  is an element $\sigma\in F$ such that $g^\sigma$ is a common co-generator of the elements of $S$ in $F$. 
%	\end{enumerate}
%\end{corollary}
%
%Note that if $f\in F$ is a non-trivial element then there is a number $\beta\in (0,1)$ that is not fixed by $f$. Hence, we have the following. 

\begin{corollary}\label{cor:uni1}
	Let $f,g\in F$ be such that $\{\pi(f),\pi(g)\}$ is a generating set of $\Ztwo$. Then there is an element $\sigma\in F$ such that $\{f,g^\sigma\}$ is a generating set of $F$. 
	%Let $c,d\in\mathbb{Z}$ be co-prime. Then for every $g\in F$ such that $\pi(g)=(c,d)$ and such that the following holds. For every $f\in F$ such that $(c,d)$ is a co-generator of $\pi(f)$ in $\Ztwo$, there exists $\sigma\in F$ such that $\la f,g^\sigma\ra =F$.
\end{corollary}

Note that Corollary \ref{cor:uni1} shows that the semi-spread of Thompson's group $F$ is $\geq 1$. Indeed, for every $(c,d)\in\Ztwo$ that is part of a generating pair of $\Ztwo$ (i.e., such that $c$ and $d$ are co-prime) and for every $g\in F$ such that $\pi(g)=(c,d)$ we have the following: for every $f\in F$ such that $\pi(g)$ is a co-generator of $\pi(f)$, the element $f$ has a co-generator in $F$ that is a conjugate of $g$. 

\begin{remark}\label{rem:uni_not2}
	Let $f_1,f_2\in F$ be elements such that $\pi(f_1)=(1,0)$, $\pi(f_2)=(0,1)$ and such that $f_1$ fixes the interval $[\frac{1}{2},1]$ pointwise and $f_2$ fixes the interval $[0,\frac{1}{2}]$ pointwise. Let $S=\{f_1,f_2\}$ and note that $(1,1)$ is a common co-generator of the elements of $\pi(S)$ in $\Ztwo$. Note also that every element $g\in F$ such that $\pi(g)=(1,1)$ fixes some number $\alpha\in (0,1)$ (see Remark \ref{rem:cd>0} below) and as such, by Theorem \ref{thm:most general}, the elements of $S$ do not have a common co-generator that is a conjugate of $g$. 
\end{remark}

Remark \ref{rem:uni_not2} shows that the uniform semi-spread of $F$ is smaller than $2$.
 Hence, the uniform semi-spread of $F$ is equal to $1$. However,  the discussion following Theorem \ref{thm:most general} implies the following. 

\begin{corollary}\label{cor:uniform}
	Let $c,d\in\mathbb{Z}$ be co-prime such that $cd\neq 1$. Then there exists $g\in F$ such that $\pi(g)=(c,d)$ and such that the following holds. For every finite subset $S\subseteq F$, such that $(c,d)$ is a common co-generator of the elements of $\pi(S)$ in $\Ztwo$, there exists an element $\sigma\in F$ such that $g^\sigma$ is a common co-generator of the elements of $S$ in $F$. In fact, if $|cd|\neq 1$, the result holds for every $g\in F$ such that $\pi(g)=(c,d)$. If $cd=-1$, the result holds for every $g\in F$ such that $\pi(g)=(c,d)$ and such that $g$ does not fix any number in $(0,1)$. 
\end{corollary}

Corollary \ref{cor:uniform} shows that (in the notations of Definition \ref{def}(2)), if we restrict our attention to elements $a=(c,d)\in\Ztwo$ such that $cd\neq 1$, then the ``uniform semi-spread'' of $F$ can be considered to be infinite.

   Note that the results discussed above  show that in a sense it is very easy to generate Thompson's group $F$.  
   Several other results demonstrate (in different ways) the abundance of generating pairs of Thompson's group $F$.
    Recall that in \cite{G-RG}, we proved that in the two  natural probabilistic models studied in \cite{CERT}, a random pair of elements of $F$ generates $F$ with positive probability. 
   In \cite{GS-set} (improving a result from \cite{GGJ}) Sapir and the author proved that Thompson's group $F$ is invariably generated by $2$ elements (i.e., there are $2$ elements $f_1,f_2,\in F$ such that regardless of how each one of them is conjugated, together they generate $F$.)  
   In 2010, Brin proved that
   the free group of rank $2$ is a limit of $2$-markings of Thompson's group $F$ in the space of all $2$-marked groups \cite{B}.
   Brin's result and Lodha's new (and much shorter) proof \cite{L} also demonstrate the abundance of generating pairs of $F$.

   {\bfseries Acknowledgments:} The author would like to thank the anonymous referee of \cite{G-32} for asking (in the terminology used in this paper) if the uniform semi-spread of $F$ is $\geq 1$.

		\section{Preliminaries}\label{s:FT}
	
	\subsection{F as a group of homeomorphisms}
	
	Recall that Thompson group $F$ is the group of all piecewise linear homeomorphisms of the interval $[0,1]$ with finitely many breakpoints where all breakpoints are  dyadic fractions and all slopes are integer powers of $2$.  
	The group $F$ is generated by two functions $x_0$ and $x_1$ defined as follows \cite{CFP}.
	
	\[
	x_0(t) =
	\begin{cases}
		2t &  \hbox{ if }  0\le t\le \frac{1}{4} \\
		t+\frac14       & \hbox{ if } \frac14\le t\le \frac12 \\
		\frac{t}{2}+\frac12       & \hbox{ if } \frac12\le t\le 1
	\end{cases} 	\qquad	
	x_1(t) =
	\begin{cases}
		t &  \hbox{ if } 0\le t\le \frac12 \\
		2t-\frac12       & \hbox{ if } \frac12\le t\le \frac{5}{8} \\
		t+\frac18       & \hbox{ if } \frac{5}{8}\le t\le \frac34 \\
		\frac{t}{2}+\frac12       & \hbox{ if } \frac34\le t\le 1 	
	\end{cases}
	\]
	
	The composition in $F$ is from left to right.

	Every element of $F$ is completely determined by how it acts on the set $\zz$. Every number in $(0,1)$ can be described as $.s$ where $s$ is an infinite word in $\{0,1\}$. For each element $g\in F$ there exists a finite collection of pairs of (finite) words $(u_i,v_i)$ in the alphabet $\{0,1\}$ such that every infinite word in $\{0,1\}$ starts with exactly one of the $u_i$'s. The action of $F$ on a number $.s$ is the following: if $s$ starts with $u_i$, we replace $u_i$ by $v_i$. For example, $x_0$ and $x_1$  are the following functions:

	\[
	x_0(t) =
	\begin{cases}
		.0\alpha &  \hbox{ if }  t=.00\alpha \\
		.10\alpha       & \hbox{ if } t=.01\alpha\\
		.11\alpha       & \hbox{ if } t=.1\alpha\
	\end{cases} 	\qquad	
	x_1(t) =
	\begin{cases}
		.0\alpha &  \hbox{ if } t=.0\alpha\\
		.10\alpha  &   \hbox{ if } t=.100\alpha\\
		.110\alpha            &  \hbox{ if } t=.101\alpha\\
		.111\alpha                      & \hbox{ if } t=.11\alpha\
	\end{cases}
	\]
	where $\alpha$ is any infinite binary word.
	
	The group $F$ has the following finite presentation \cite{CFP}.
	$$F=\la x_0,x_1\mid [x_0x_1^{-1},x_1^{x_0}]=1,[x_0x_1^{-1},x_1^{x_0^2}]=1\ra.$$ %where $a^b$ denotes $b^{-1} ab$.

	\subsection{Elements of F as pairs of binary trees} \label{sec:tree}
	
	Often, it is more convenient to describe elements of $F$ using pairs of finite binary trees (see \cite{CFP} for a detailed exposition). The considered binary trees are rooted \emph{full} binary trees; that is, each vertex is either a leaf or has two outgoing edges: a left edge and a right edge. A  \emph{branch} in a binary tree is a simple path from the root to a leaf. If every left edge in the tree is labeled ``0'' and every right edge is labeled ``1'', then a branch in $T$ has a natural binary label. We rarely distinguish between a branch and its label. A branch in a binary tree is an \emph{inner} branch if it is not the left-most nor the right-most branch of $T$ (i.e., if its label contains both digits ``0'' and ``1'').
	
	Let $(T_+,T_-)$ be a pair of finite binary trees with the same number of leaves. The pair $(T_+,T_-)$ is called a \emph{tree-diagram}. Let $u_1,\dots,u_n$ be the (labels of) branches in $T_+$, listed from left to right. Let $v_1,\dots,v_n$ be the (labels of) branches in $T_-$, listed from left to right. We say that the tree-diagram $(T_+,T_-)$ has the \emph{pair of branches} $u_i\rightarrow v_i$ for $i=1,\dots,n$. The tree-diagram $(T_+,T_-)$ \emph{represents} the function $g\in F$ which takes binary fraction $.u_i\alpha$ to $.v_i\alpha$ for every $i$ and every infinite binary word $\alpha$. We also say that the element $g$ takes the branch $u_i$ to the branch $v_i$.
	For a finite binary word $u$, we denote by $[u]$  the dyadic interval $[.u,.u1^{\mathbb{N}}]$ and by $[u)$ the interval $[.u,.u1^\mathbb{N})$. If $u\rightarrow v$ is a pair of branches of $(T_+,T_-)$, then $g$ maps the interval $[u]$ linearly onto $[v]$. 
	
	A \emph{caret} is a binary tree composed of a root with two children. If $(T_+,T_-)$ is a tree-diagram and one attaches a caret to the $i^{th}$ leaf of $T_+$ and the $i^{th}$ leaf of $T_-$ then the resulting tree diagram is \emph{equivalent} to $(T_+,T_-)$ and represents the same function in $F$. The opposite operation is that of \emph{reducing} common carets. A tree diagram $(T_+,T_-)$ is called \emph{reduced} if it has no common carets; i.e, if there is no $i$ for which the  $i$ and ${i+1}$ leaves of both $T_+$ and $T_-$ have a common father. Every tree-diagram is equivalent to a unique reduced tree-diagram. Thus elements of $F$ can be represented uniquely by reduced tree-diagrams \cite{CFP}.
	The reduced tree-diagrams of the generators $x_0$ and $x_1$ of $F$ are depicted in Figure \ref{fig:x0x1}.
	
	\begin{figure}[ht]
		\centering
		\begin{subfigure}{.55\textwidth}
			\centering
			\includegraphics[width=.55\linewidth]{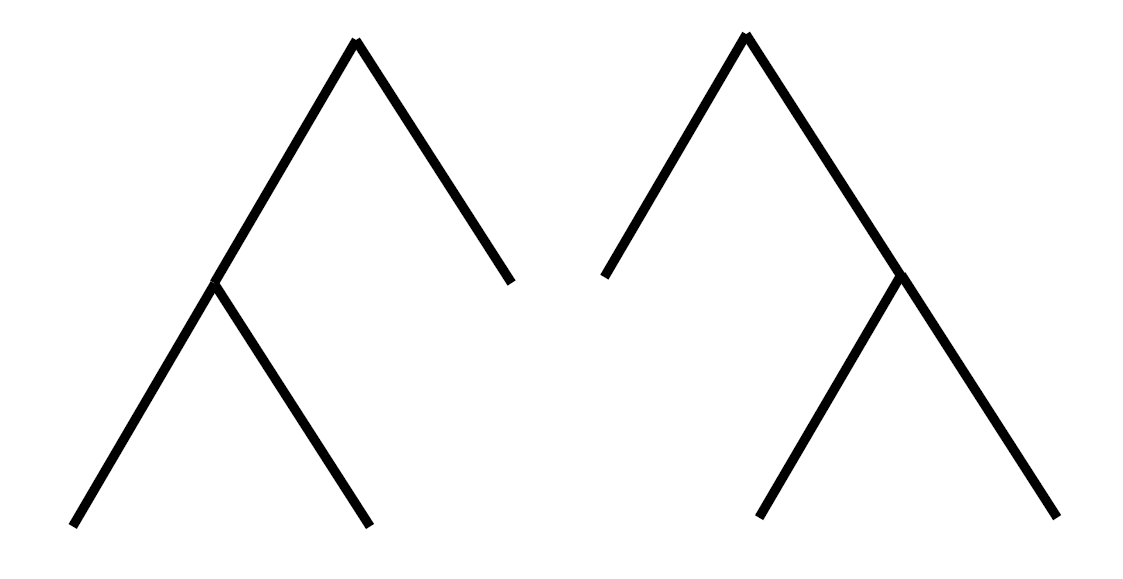}
			\caption{}
			\label{fig:x0}
		\end{subfigure}%
		\begin{subfigure}{.55\textwidth}
			\centering
			\includegraphics[width=.55\linewidth]{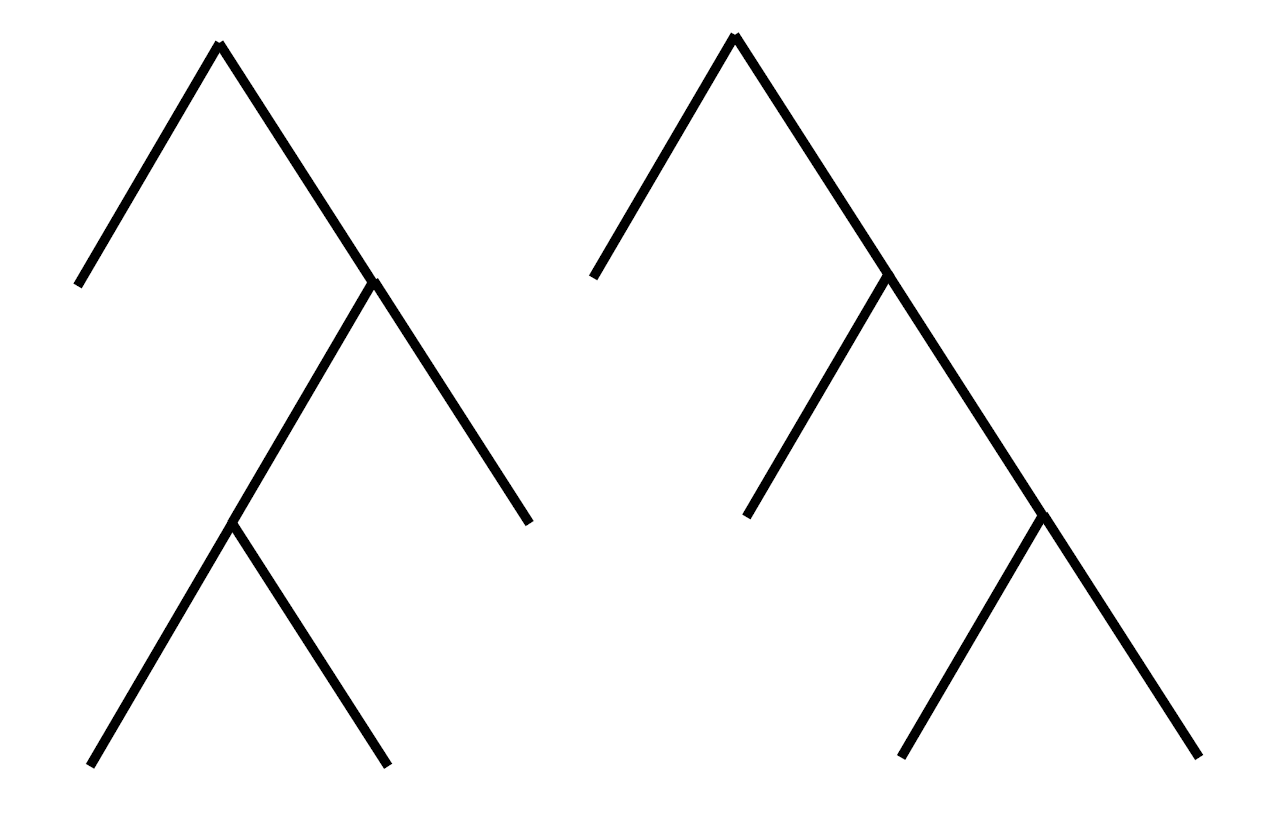}
			\caption{}%The tree diagram $(T_+,T_-)$ of $x_1$}
			\label{fig:x1}
		\end{subfigure}
		\caption{(A) The reduced tree-diagram of $x_0$. (B) The reduced tree-diagram of $x_1$. In both figures, $T_+$ is on the left and $T_-$ is on the right.}
		\label{fig:x0x1}
	\end{figure}
	
	When we say that a function $f\in F$ has a pair of branches $u_i\rightarrow v_i$, the meaning is that some tree-diagram representing $f$ has this pair of branches. In other words, this is equivalent to saying that $f$ maps the dyadic interval $[u_i]$ linearly onto $[v_i]$.
	Clearly, if $u\rightarrow v$ is a pair of branches of $f$, then for any finite binary word $w$, $uw\rightarrow vw$ is also a pair of branches of $f$. Similarly, if $f$ has the pair of branches $u\rightarrow v$ and $g$ has the pair of branches $v\rightarrow w$ then $fg$ has the pair of branches $u\rightarrow w$.

\subsection{Choosing elements of $F$}

Let $I_1,I_2\subseteq [0,1]$. We will write $I_1\prec I_2$ if for every $x\in I_1$ and $y\in I_2$, we have $x\leq y$. %In other words, we write that $I_1\prec I_2$ if the interval $I_2$ is to the left of the interval $I_2$ (and the intervals have at mosy ) 
In most proofs in this paper, we choose elements with a given set of pairs of branches, or elements which map certain sub-intervals of $[0,1]$ in a predetermined way. In doing so, we usually apply the next lemma or the following remark, often without an explicit reference. 

\begin{lemma}\label{choice}
	Let $[a_1,b_1],\dots,[a_m,b_m]$, $[c_1,d_1],\dots,[c_m,d_m]$ be closed subintervals of $[0,1]$ with endpoints from $\zz$. 
	Assume that $[a_1,b_1]\prec \dots\prec [a_m,b_m]$
	and that $[c_1,d_1]\prec\dots\prec [c_m,d_m]$.
%	Assume that the interiors of the intervals $[a_1,b_1],\dots,[a_m,b_m]$ (resp. $[c_1,d_1],\dots,[c_m,d_m]$) are pairwise disjoint and that the intervals $[a_1,b_1],\dots,[a_m,b_m]$ (resp. $[c_1,d_1],\dots, [c_m,d_m]$), considered as sub-intervals of $[0,1]$, are ordered from left to right.
 Assume in addition that the following conditions are satisfied.
	\begin{enumerate}
		%\item $[a_i,b_i]$ has empty interior if and only if $[c_i,d_i]$ has empty interior.
		\item[(1)] $0\in [a_1,b_1]$ if and only if $0\in [c_1,d_1]$.
		\item[(2)] $1\in [a_m,b_m]$ if and only if $1\in [c_m,d_m]$.
		\item[(3)] for all $i=1,\dots,m-1$, the intervals $[a_i,b_i]$ and $[a_{i+1},b_{i+1}]$ share a boundary point if and only if the intervals $[c_i,d_i]$ and $[c_{i+1},d_{i+1}]$ share a boundary point.
	\end{enumerate}
	Then there is an element $f\in F$ which maps each interval $[a_i,b_i]$, $i=1,\dots,m$ onto the interval $[c_i,d_i]$. In addition, if for some $i$,  $\frac{b_i-a_i}{d_i-c_i}$ is an integer power of $2$ (in particular, if both $[a_i,b_i] $ and $[c_i,d_i]$ are dyadic intervals), then $f$ can be taken to map
	$[a_i,b_i]$ linearly onto $[c_i,d_i]$.
\end{lemma}

The lemma follows directly from the proof of \cite[Lemma 2.1]{BrinCh}.

\begin{remark}\label{r:choice}
	The proof of Lemma 2.1 in \cite{BrinCh} also implies (in the notations of Lemma \ref{choice}) that if for each $i$ we choose an element $g_i\in F$ which maps $[a_i,b_i]$ onto $[c_i,d_i]$, then there is an element $f\in F$ such that for all $i\in\{1,\dots,m\}$, the restriction of $f$ to $[a_i,b_i]$ coincides with the restriction of $g_i$ to that interval.
\end{remark}

Let $g\in F$. Recall that an interval $(a,b)$ is an \emph{orbital} of $g$ if $g$ fixes the endpoints $a$ and $b$ and does not fix any number in $(a,b)$. Note that if $(a,b)$ is an orbital of $g$, then for every $x\in (a,b)$, we have $g(x)\in (a,b)$. %Equivalently, $(a,b)$ is an orbital of $g$ if it is the convex hull of the action of $
Now, assume that $(a,b)$ is an orbital of $g$. Then the graph of the function $g$ (drawn in an $xy$ coordinate system)  intersects the diagonal $y=x$ at $x=a$ and at $x=b$ and does not intersect the diagonal $y=x$ at any $x\in (a,b)$. Hence, in the interval $(a,b)$ the graph of the function is either entirely above the diagonal or entirely below the diagonal. In other words, %if $(a,b)$ is an orbital of $h$, then
either for every $x\in (a,b)$ we have $g(x)>x$ or for every $x\in (a,b)$ we have $g(x)<x$. In the first case, we say that $(a,b)$ is a \emph{push-up} orbital of $g$, while in the second case we say that $(a,b)$ is a \textit{push-down} orbital of $g$. Note that $(a,b)$ is a push-up orbital of $g$ if and only if $g'(a^+)>1$ if and only if $g'(b^-)<1$. Equivalently, $(a,b)$ is a push-down orbital of $g$ if and only if $g'(a^+)<1$ if and only if $g'(b^-)>1$. In particular, if $(a,b)$ is an orbital of $g$ and $c=\log_2(g'(a^+))$ and $d=\log_2(g'(b^-))$, then $cd<0$.

 Recall that if $f\in F$ then $\pi(f)=(\log_2(f'(0^+)),\log_2(f'(1^-)))$. Hence, we have the following. 

\begin{remark}\label{rem:cd>0}
	Let $f\in F$ be such that $\pi(f)=(c,d)$. If $cd\geq 0$, then the interval $(0,1)$ is not an orbital of $f$. In other words, if  $cd\geq 0$, then $f$ must fix some number in $(0,1)$. 
\end{remark}

Note that the following well-known remark also holds (one can prove it, for example, using Remark \ref{r:choice}).

\begin{remark}\label{rem:cd<0}
	Let $c,d\in\mathbb{Z}$ be such that $cd<0$. Then there exists $f\in F$ such that $\pi(f)=(c,d)$ and such that $f$ does not fix any number in $(0,1)$. 
\end{remark}

The following remarks about orbitals will be used bellow without reference. For proofs and more information on orbitals see, for example, \cite{B-geom}.

\begin{remark}
	Assume that an element $g\in F$ has a push-up (resp. push-down) orbital $(a,b)$ and let $\sigma\in F$. Then $(\sigma(a),\sigma(b))$ is a push-up (resp. push-down) orbital of $g^\sigma$.
\end{remark}

\begin{remark}\label{rem:orb_n}
	Let $(a,b)$ be an orbital of $g\in F$ and let $x,y\in (a,b)$. Then there exists $n\in\mathbb{Z}$ such that $f^n(x)>y$
\end{remark}

%Note also that if $\sigma\in F$ and $(a,b)$ is a push-up (resp. push-down) orbital of $g$, then  $(\sigma(a),\sigma(b))$ is a push-up (resp. push-down) orbital of $g^\sigma$. For more information about orbitals see, for example,  \cite{B-geom}.
%Note that if $(a,b)$ is an orbital of $h$

The next lemma will be used below for choosing a conjugate of a given element while ensuring that this conjugate has certain pairs of branches. %certain properties are satisfied. 
The lemma is similar to Lemma 2.8 in \cite{BHS} for choosing conjugates in Thompson's group $T$.

\begin{lemma}\label{lem:con}
	Let $f,g\in F$ and assume that $(a,b)$ and $(c,d)$ are  orbitals of $f$ and $g$, respectively, such  that $(a,b)$ and $(c,d)$ are both push-up orbitals or both push-down orbitals. Let $[x,y]$ be a closed sub-interval of $(a,b)$. Then there exists a function $\sigma\in F$ such that $[x,y]\subseteq (\sigma(c),\sigma(d))$ and such that	
	$g^\sigma$ coincides with $f$ on the interval $[x,y]$. Assume in addition that there  is a closed interval $J$ with dyadic endpoints %disjoint from $(a,b)\cup (c,d)$ and
	such that $J\prec (a,b)$ and $J\prec (c,d)$. Then $\sigma$ can be taken to be the identity on $J$. 
	
	%	and assume that the closed interval $[x,y]$ is contained inside some push-up orbital of $f$. 	%Let $f,g\in F$ an assume that $f$ has an orbital $(\alpha,\beta)$ and $g$ has an orbital $(\gamma,\delta)$.
	%	%	Let $[x,y]$ be a closed sub-interval of $(\alpha,\beta)$. 
	%	Then there exists a function $\sigma\in F$ such that $g^{\sigma}$ coincides with $f$ or with $f^{-1}$ on the interval $[x,y]$.
	%Let $x\in (a,b)$ and $y\in (\gamma,\delta)$ be dyadic fractions. Then for every $n\in\mathbb{N}$
\end{lemma}

\begin{proof}
	The proof is almost identical to the beginning of the proof of \cite[Lemma 2.8]{BHS}. We will prove the lemma in the case were both $(a,b)$ and $(c,d)$ are push-up orbitals. 	First, we can assume that $x$ is dyadic. Indeed, if it is not dyadic, we can  replace it by some dyadic number in $(a,x)$. Now, since $x,y\in (a,b)$, and $(a,b)$ is a push-up orbital of $f$, by Remark \ref{rem:orb_n}, there exists $n\in\mathbb{N}$ such that $f^n(x)>y$. %(see, for example, \cite{B-geom}).
	 Let $z=f^n(x)$. We will prove that one can define an element $\sigma\in F$ such that $g^\sigma$  coincides with $f$  on the interval $[x,z]$ (and in particular, on the interval $[x,y]$). 
	Let $\alpha\in (c,d)$ be some dyadic fraction. 
	Since $(c,d)$ is a push-up orbital of $g$, the intervals $$[\alpha, g(\alpha)], [g(\alpha),g^2(\alpha)], \dots [g^{n}(\alpha), g^{n+1}(\alpha)]$$
	are consecutive sub-intervals of $(c,d)$ with dyadic endpoints. Similarly, since $(a,b)$ is a push-up orbital of $f$, the intervals 
	$$[x,f(x)],[f(x),f^2(x)],\dots,[f^n(x),f^{n+1}(x)]$$
	are consecutive sub-intervals of $(a,b)$ with dyadic endpoints.
	To define $\sigma$, we first define inductively $n+1$ functions $\sigma_0,\sigma_1\dots,\sigma_n\in F$ such that for each $i\in\{0,\dots,n\}$ the function $\sigma_i$ maps the interval $[g^i(\alpha),g^{i+1}(\alpha))]$ onto the interval $[f^i(x),f^{i+1}(x)]$. 
	For $i=0$, we let $\sigma_0$ be some function in $F$ which maps the interval $[\alpha,g(\alpha)]$ onto the interval $[x,f(x)]$. 
	Now, for each $i=1,\dots,n$, we let $\sigma_i$  be some function in $F$ such that on the interval $[g^i(\alpha),g^{i+1}(\alpha)]$ it coincides with $g^{-1}\sigma_{i-1}f$. 
	
	Note that for each $i\in\{0,\dots,n\}$, the function $\sigma_i$ indeed maps the interval $[g^i(\alpha),g^{i+1}(\alpha)]$ onto the interval $[f^i(x),f^{i+1}(x)]$. Hence, by Remark \ref{r:choice}, we can define a function $\sigma\in F$ such that for each $i\in\{0,\dots,n\}$, the function $\sigma$ coincides on the interval $[g^i(\alpha),g^{i+1}(\alpha)]$ with the function $\sigma_i$. Remark \ref{r:choice} also implies that if there is an interval $J$ as described, then $\sigma$ can be taken to be the identity on $J$. 
	
	Note that by definition, $\sigma$ maps the interval $[\alpha,g^{n+1}(\alpha)]\subseteq (c,d)$ onto the interval $[x,z]\supseteq [x,y]$. Hence, $[x,y]$ is contained in $(\sigma(c),\sigma(d))$. It remains to note that the function  $g^{\sigma}$ coincides with $f$ on the interval $[x,z]=[x,f^n(x)]$, as required. Indeed, let $\beta\in[x,f^n(x)]$, then there exists $i\in\{0,\dots,n-1\}$ such that $\beta\in [f^i(x),f^{i+1}(x)]$. Then (as explained below) we have the following equality.
	
	\begin{equation*}
		\begin{split}
			g^\sigma(\beta) &  =\sigma(g(\sigma^{-1}(\beta)))\\
			& 	= \sigma (g (\sigma_i^{-1}(\beta)) )\\
			&	=\sigma_{i+1}(g (\sigma_i^{-1}(\beta)) )\\
			&	=f(\sigma_{i}(g^{-1}(g(\sigma_i^{-1}(\beta)))))\\
			&=f(\beta)
		\end{split}
	\end{equation*}
	
	Indeed, the second equality  is due to the fact that $\beta\in [f^i(x),f^{i+1}(x)]$.
	Now, since $\beta\in [f^i(x),f^{i+1}(x)]$, we have $\sigma_i^{-1}(\beta)\in [g^i(\alpha),g^{i+1}(\alpha)]$. Hence, $g(\sigma_i^{-1}(\beta))\in [g^{i+1}(\alpha),g^{i+2}(\alpha)]$, which implies the third equality. Finally, the  fourth equality follows from the definition of $\sigma_{i+1}$ on the interval $[g^{i+1}(\alpha),g^{i+2}(\alpha)]$.
	% that, which implies the equality on the third line. 
	%	
	%	the function $\sigma$ is as required. Indeed,
	%	
	%	
	%%	 for each $i\in\{0,\dots,n-1\}$, on the interval $[f^i(x),f^{i+1}(x)]$ we have 
	%	$$g^{\sigma}|_{[f^i(x),f^{i+1}(x)]}=
	%	\sigma^{-1}|_{[f^i(x),f^{i+1}(x)] g|_{[g^i(x),g^{i+1}(x)]}
	%		 \sigma|_{[g^{i+1}(x),g^{i+2}(x)]}=
	%		 \sigma_{i}^{-1}|_{[f^i(x),f^{i+1}(x)]g|_{[g^i(x),g^{i+1}(x)]}\sigma_{i+1}|_{[g^{i+1}(x),g^{i+2}(x)]}=
	%		 	\sigma_{i}^{-1}|_{[f^i(x),f^{i+1}(x)]g|_{[g^i(x),g^{i+1}(x)]}g^{-1}|_{[g^i(x),g^{i+1}(x)]}\sigma_{i}|_{[g^{i}(x),g^{i+1}(x)]}f|_{[f^i(x),f^{i+1}(x)]=f|_{[f^i(x),f^{i+1}(x)]$$

	%	f^{-1}|_{[f^i(x),f^{i+1}(x)]}\sigma_{i-1}^{-1}|_{[f^{i-1}(x),f^{i}(x)]}g|_{[g^{i-1}(\alpha),g^{i}(\alpha)]}\sigma_i|_{[g^i(\alpha),g^{i+1}(\alpha)]}f|_{[f^i(x),f^{i+1}(x)]}$$

\end{proof}

	\subsection{Generating Thompson's group $F$}
	
	Let $H$ be a subgroup of $F$. A function $f\in F$ is said to be a \emph{piecewise-$H$} function if there is a finite subdivision of the interval $[0,1]$ such that on each interval in the subdivision, $f$ coincides with some function in $H$. Note that since all breakpoints of elements in $F$ are dyadic fractions, a function $f\in F$ is a piecewise-$H$ function if and only if there is a  dyadic subdivision of the interval $[0,1]$ into finitely many pieces such that on each dyadic interval in the subdivision, $f$ coincides with some function in $H$. %Equivalently, a function $f\in F$ is a  piecewise-$H$ function if and only if it has a (not necessarily reduced) tree-diagram $(T_+,T_-)$ such that each pair of branches $u\to v$ of $(T_+,T_-)$ is a pair of branches of some element in $H$. 

	Following \cite{GS,G16}, we define the \emph{closure} of a subgroup $H$ of $F$, denoted $\Cl(H)$, to be the subgroup of $F$ of all piecewise-$H$ functions. A subgroup $H$ of $F$ is \emph{closed} if $H=\Cl(H)$.  In \cite{G16}, the author proved that the generation problem in $F$ is decidable. That is, there is an algorithm that decides given a finite subset $X$ of $F$ whether it generates the whole $F$. A recent improvement of \cite[Theorem 7.14]{G16} is the following theorem. 
	
	\begin{theorem}\cite[Theorem 1.3]{G-max}\label{gen}
		Let $H$ be a subgroup of $F$. Then $H=F$ if and only if the following conditions are satisfied. 
		\begin{enumerate}
			\item[(1)] $H[F,F]=F$.
			\item[(2)] $[F,F]\leq \Cl(H)$. %contains the derived subgroup of $F$. 
			%\item[(3)] There is an element $h\in H$ which fixes a finite dyadic fraction $\alpha\in (0,1)$ such that $h'(\alpha^-)=1$ and 
			%$h'(\alpha^+)=2$.
		\end{enumerate}
	\end{theorem}
	
	Recall that %the derived subgroup
	 $[F,F]$
	is the kernel of the abelianization map $\pi\colon F\to \mathbb{Z}^2$.
	%	maps each function $f\in F$ to $$\pi(f)=(\log_2(f'(0^+)),\log_2(f'(1^-))).$$
	%The derived subgroup $[F,F]$ is the kernel of $\pi$. 
	Hence, a subgroup $H\leq F$ satisfies Condition (1) of Theorem \ref{gen} if and only if $\pi(H)=\mathbb{Z}^2$. 
	Note also that as the kernel of $\pi$, the derived subgroup of $F$ 	 can be characterized as the subgroup of $F$ of all functions $f$ with slope $1$ both at $0^+$ and at $1^-$.
	That is, a function $f\in F$ belongs to $[F,F]$ if and only if the reduced (equiv. any) tree-diagram of $f$ has pairs of branches of the form $0^m\rightarrow 0^m$ and $1^n\rightarrow 1^n$ for some $m,n\in\mathbb{N}$.

	\section{Subgroups $H$ of $F$ whose closure contains $[F,F]$}
	
	In the next section, we apply Theorem \ref{gen} to prove that a given subset of $F$ generates $F$. %a normal subgroup $H$ of $F$. % that contains $[F,F]$.
	To do so, we have to prove in particular that the subgroup $H$ generated by that subset satisfies Condition (2) of the theorem, i.e., that the closure of $H$ contains the derived subgroup of $F$. To that end, with each subgroup of $F$, we associate an equivalence relation on the set of finite 
	binary words.
	
	Let us denote by $\mathcal B$  the set of all finite binary words and let $\mathcal B'$ be the set of all finite binary words which contain both digits ``0'' and ``1'' (i.e., of all finite binary words $u$ such that the interval $[u]$ is contained in $(0,1)$).
	For $u,v\in\mathcal B$ we say that $v$ is a \emph{descendant} of $u$ if $u$ is a strict prefix of $v$. We say that $u$ and $v$ are \emph{incomparable} if $u$ is not a prefix of $v$ and $v$ is not a prefix of $u$. Note that if $u$ and $v$ are incomparable then the interiors of $[u]$ and $[v]$ have empty intersection.  In that case, either $[u]\prec [v]$ or $[v]\prec [u]$.

	\begin{definition}
		Let $H$ be a subgroup of $F$.
		The \emph{equivalence relation induced by $H$} on the set of finite binary words $\mathcal B$, denoted $\sim_H$, is defined as follows. For every pair of finite binary words $u,v\in\mathcal B$ we have $u\sim_H v$ if and only if there is an element in $H$ with the pair of branches $u\to v$.
	\end{definition}
	
	Note that if $u\sim_Hv$ then for every finite binary word $w$, we have $uw\sim_H vw$. 
	Moreover, if $H$ is closed, we have the following. 
	
	\begin{lemma}[{\cite[Lemma 10]{G-32}}]
		\label{coherent}
		Let $H$ be a closed subgroup of $F$. Then for every pair of finite binary words $u,v\in\mathcal B$ we have $u\sim_H v$ if and only if $u0\sim_H v0$ and $u1\sim_H v1$. 
	\end{lemma}

	\begin{remark}
		An equivalence relation $\sim$ on the set of finite binary words $\mathcal B$ such that for every $u,v\in\mathcal B$ we have $u\sim v$ if and only if $u0\sim v0$ and $u1\sim v1$ is called a \emph{coherent} equivalence relation in \cite{BBQS}. Coherent equivalence relations were used in \cite{BBQS} to study maximal subgroups of Thompson's group $V$. 
	\end{remark}
	
	The following corollary follows from Lemma \ref{coherent} by induction on $k$. 
	
	\begin{corollary}\label{cor:coherent}
		Let $H$ be a closed subgroup of $F$. Let $u,v\in\mathcal B$ and let $k\in\mathbb{N}$. Assume that for every finite binary word $w$ of length $k$ we have $uw\sim_H vw$. Then $u\sim_H v$. 
	\end{corollary}

	Recall that $\mathcal B'$ is the set of all finite binary words which contain both digits ``0'' and ``1''. %In other words, $\mathcal B'=\mathcal B\setminus\{\emptyset, 0^n,1^n\mid n\in\mathbb{N}\}$. 
	
	\begin{lemma}\label{lem:all inner words identified}
		Let $H$ be a closed subgroup of $F$. Assume that for every pair of finite binary words $u,v\in\mathcal B'$ we have $u\sim_H v$. Then $H$ contains the derived subgroup of $F$. 
	\end{lemma}
	
	\begin{proof}
		Let $f\in [F,F]$ . Then the reduced tree-diagram of $f$ consists of pairs of branches 
		\[
		f :
		\begin{cases}
			0^m & \rightarrow 0^m\\
			u_i  & \rightarrow v_i \mbox{ for } i=1,\dots,k \\
			1^n & \rightarrow 1^n\\
		\end{cases}
		\]
		where $k,m,n\in\mathbb{N}$ and where for each $i=1,\dots,k$, the binary words $u_i$ and $v_i$ both %contain both digits $``0"$ and $``1"$ and as such 
		belong to $\mathcal B'$. By assumption, for each $i=1,\dots,k$ we have $u_i\sim_H v_i$ and as such there is an element $h_i\in H$ with the pair of branches $u_i\rightarrow v_i$. Then $h_i$ coincides with $f$ on the interval $[u_i]$. We note also that $f$ coincides with the identity function {\bfseries{1}} $\in H$ on $[0^m]$ and on $[1^n]$. Since $[0^m],[u_1],\dots,[u_k],[1^n]$ is a subdivision of the interval $[0,1]$ and on each of these intervals $f$ coincides with a function in $H$, $f$ is a piecewise-$H$ function. Since $H$ is closed, $f\in H$. 
	\end{proof}

Due to Lemma \ref{lem:all inner words identified}, we will be interested in closed subgroups $H$ of $F$ such that all finite binary words in $\mathcal B'$ are $\sim_H$-equivalent. The following lemmas will be useful.
%In other words, we will be interested in closed subgroups $H$ of $F$ such that for every pair of finite binary words $u$ and $v$ such that 

\begin{lemma}\label{lem:clear}
	Let $H$ be a closed subgroup of $F$ and let $u$ be a finite binary word such that $u\sim_H u0\sim_H u1$. Let $v$ be a finite binary word and assume that there exists $k\in\mathbb{N}$ such that for every finite binary word $s$ of length $k$, we have $vs\sim_H u$. Then $v\sim_H u$. 
\end{lemma}

\begin{proof}
	First, note that since $u\sim_H u0\sim_H u1$, it follows by induction that for every finite binary word $s$, we have $us\sim_H u$. In other words, every descendant  of $u$ is $\sim_H$-equivalent to $u$. Now, by assumption, for every finite binary word $s$ of length $k$, we have $vs\sim_H u$. Hence, for every finite binary word $s$ of length $k$ we have $vs\sim_H us$. Therefore, by Corollary \ref{cor:coherent}, we have $v\sim_H u$, as required. 
\end{proof}

\begin{lemma}\label{lem:ezer}
	Let $H$ be a closed subgroup of $F$ and let $u$ be a finite binary word such that $u\sim_H u0\sim_H u1$. Let $h\in H$ and let $[a,b]\subseteq (0,1)$ be such that $h([u])=[a,b]$. Then for every finite binary word $v$ such that $[v]$ is contained in $[a,b]$, we have $v\sim_H u$. 
\end{lemma}

\begin{proof}
	Let $v$ be a finite binary word such that $[v]\subseteq [a,b]$. 
	Note that $h^{-1}([v])$ is contained in $[u]$. 
	Since $h^{-1}\in F$, there exists some $k\in\mathbb{N}$ such that for every finite binary word $s$ of length $k$, the function $h^{-1}$ acts linearly on the dyadic interval $[vs]$ and as such, maps it onto some dyadic interval contained in $[u]$. Hence, for each  binary word $s$ of length $k$, there exists some finite binary words $w_s$ such that $h^{-1}$ has the pair of branches $vs\to uw_s$. From the assumption on $u$, it follows that every descendant of $u$ is $\sim_H$-equivalent to $u$. Hence, for every finite binary word $s$ of length $k$, we have $vs\sim_H uw_s\sim_H u$. Hence, by Lemma \ref{lem:clear}, we have $v\sim_H u$, as required.
\end{proof}

%Let $h\in F$. Recall that an interval $(a,b)$ is an orbital of $h$ if $h$ fixes the endpoints $a$ and $b$ and does not fix any number in $(a,b)$. %Now, assume that $(a,b)$ is an orbital of $h$. Then the graph of the function $h$ (drawn in an $xy$ coordinate system)  intersects the diagonal $y=x$ at $x=a$ and at $x=b$ and does not intersect the diagonal $y=x$ at any $x\in (a,b)$. Hence, in the interval $(a,b)$ the graph of the function is either entirely above the diagonal or entirely below the diagonal (see Figure ). In other words, %if $(a,b)$ is an orbital of $h$, then
% either for every $x\in (a,b)$ we have $h(x)>x$ or for every $x\in (a,b)$ we have $h(x)<x$. In the first case, we say that $(a,b)$ is a \emph{push-up} orbital of $h$, while in the second case we say that $(a,b)$ is a push-down orbital of $h$. Note that $(a,b)$ is a push-up orbital of $h$ if and only if $h'(a^+)>1$ if and only if $h'(b^-)<1$. %Similarly, $(a,b)$ is a push-down orbital of $h$ if and only if $h'(a^-)<1$ if and only if $h'(b^+)>1$
%For more information about orbitals see, for example,  \cite{}.
%Note that if $(a,b)$ is an orbital of $h$
Let $h\in F$ and assume that $(a,b)$ is an orbital of $h$. 
%Now, let $(a,b)$ be an orbital of $h$. 
An interval $I\subseteq (a,b)$ is called a \emph{fundamental domain} of $h$ if it contains exactly one point of each orbit of the action of $\la h \ra$ on $(a,b)$. Note that if $x\in (a,b)$ and $(a,b)$ is a push-up orbital of $h$ then the interval $I=[x,h(x))$ is a {fundamental domain} of $h$. Similarly, if $(a,b)$ is a push-down orbital of $h$ and $x\in (a,b)$ then $[h(x),x)$ is a fundamental domain of $h$.
Note also that if $I\subseteq (a,b)$ is a fundamental domain of $h$ then
$(a,b)=\bigcup_{n\in\mathbb{Z}} h^n(I)$ and that every closed sub-interval of $(a,b)$ is contained in some finite sub-union.

\begin{lemma}\label{lem:fun_dom}
	Let $H$ be a closed subgroup of $F$ and let $h\in H$. Assume that $(a,b)$ is an orbital of $h$ and that  the interval $[u)=[.u,.u1^\mathbb{N})$ % %\subseteq (a,b)$ is a
	is contained in $(a,b)$ and is a fundamental domain of $h$. Assume also that $u\sim_H u0\sim_H u1$. 
	Then for every finite binary word $v$ such that $[v]$ is contained in $(a,b)$, we have $v\sim_H u$. 
\end{lemma}

\begin{proof}
	Since $[u)\subseteq (a,b)$ is a fundamental domain of $h$, we have  %the union of intervals 
	$\bigcup_{n\in\mathbb{Z}} h^n([u])=(a,b).$ 
	Let $v$ be a finite binary word such that $[v]\subseteq (a,b)$. % and note that $[v]$ is contained in the infinite union of $h^{n}([u])$ for $n\in\mathbb{Z}$. It is easy to see, that since $[v]$ is closed (and does not contain $a$ and $b$), it must be contained inside some finite union of $h^n([u])$, for some integers $n\in\mathbb{Z}$. Hence, there  
%	Since $a,b\notin [v]$, the interval $[v]$ must be contained inside some finite union 
	 If there exists $n\in\mathbb{Z}$ such that $[v]$ is contained in $h^n([u])$, then by Lemma \ref{lem:ezer}, we have $v\sim_H u$, as required. Hence, we can assume that $[v]$ is not contained in any of the intervals $h^n([u])$, $n\in\mathbb{Z}$. In that case, since $[v]$ is a closed  sub-interval of the orbital $(a,b)$, it is contained in some finite sub-union of $\bigcup_{n\in\mathbb{Z}} h^n([u])$. Since $[v]$ is a dyadic interval and each of the intervals $h^n([u])$ has dyadic endpoints,  there exists some dyadic subdivision of $[v]$, such that  each interval in the subdivision is contained inside $h^n([u])$ for some $n\in\mathbb{Z}$. Hence, there exists some $k\in\mathbb{N}$ such that for each finite binary word $s$ of length $k$, the interval $[vs]$ is contained inside $h^{n_s}([u])$ for some $n_s\in\mathbb{Z}$. Then Lemma \ref{lem:ezer} implies  that for each binary word $s$ of length $k$ we have $vs\sim_H u$. Hence, by Lemma \ref{lem:clear}, we have $v\sim_H u$, as required. 
\end{proof}

\section{Preliminary lemmas}

For the proof of the main result we will need the following preliminary lemmas.

\begin{lemma}\label{lem:condition 2}
	Let $c,d\in \mathbb{Z}$ be co-prime such that $|cd|\neq 1$. 
	Let $S$ be a finite subset of $F$ such that $(c,d)$ is a common co-generator of the elements of $S$ in $F$. Then there exists $\alpha\in (0,1)$ that is not fixed by any of the elements of $S$. 
\end{lemma}

\begin{proof}
	Since $|cd|\neq 1$, we must have $|c|\neq 1$ or $|d|\neq 1$. 
	Assume without loss of generality that $|c| \neq 1$. Let $f$ be a function in $S$ and assume that $\pi(f)=(a,b)$. Since $(c,d)$ is a co-generator of $\pi(f)$ in $\Ztwo$, we have $|ad-bc|=1$. 	
	We claim that $a\neq 0$. 
	%	 Let $f\in S$ and assume that $\pi(f)=(a,b)$. 	 Then $a\neq 0$.
	 Indeed, if $a=0$, it follows that $|bc|=1$ and therefore $|c|=1$, in contradiction to the assumption. Hence, $a\neq 0$. 
 It follows that $f'(0^+)=2^a\neq 1$.
 Since the slope of $f$ at $0^+$ is non-trivial, there exists some $\beta\in (0,1)$ such that $f$ acts linearly with non-trivial slope in the interval $(0,\beta]$ and hence does not fix any number in that interval. 
	Since $S$ is a finite set, there exists a small enough interval $(0,\gamma)$ such that  every function in $F$ does not fix any number in $(0,\gamma)$. 
\end{proof}

\begin{lemma}\label{lem:obvious}
	Let $S=\{f_1,\dots,f_k\}$ be a subset of $F$ and assume that there is a point $\alpha \in (0,1)$ that is not fixed by any of the elements in $S$.
		Then there exist incomparable finite binary words $u_1,v_1,\dots, u_k,v_k,w\in \mathcal B'$ such that 
		$$[u_1],[v_1],\dots,[u_{k}],[v_{k}]\prec[w]$$
		and such that for each $i\in\{1,\dots,k\}$ we have $u_i\sim_{\la f_i \ra} v_i \sim_{\la f_i \ra} w$.
	\end{lemma}

	\begin{proof}
		Since each of the elements $f_1,\dots,f_k$ does not fix $\alpha$, for each $i\in\{1,\dots,k\}$ the element $f_i$ has some orbital $(a_i,b_i)$ which contains $\alpha$. 
			By changing the order of the elements $f_1,\dots,f_k$ if necessary, we can assume  that  $a_k\leq a_{k-1}\leq\dots\leq a_1$.		
		Let $(a,b)=\bigcap_{i=1}^k (a_i,b_i)$ and note that $(a,b)$ is a non-empty open interval and that $a_i\leq a$ and  $b\leq b_i$ for each $i\in\{1,\dots,k\}$.
		Let $I$ be a closed sub-interval of $(a,b)$.  
		We define integers $n_1,m_1,\dots,n_k,m_k\in\BZ$  and closed intervals $I_1,J_1,\dots,I_k,J_k$
	 inductively as follows. 
		For $i=1$, since $I$ is a closed sub-interval of the orbital $(a_1,b_1)$ of $f_1$, there exists $n_1\in\mathbb{Z}$ such that $f_1^{n_1}(I)\prec I$. We let $I_1=f_1^{n_1}(I)$ and note that $I_1$ and $I$ are both contained in the orbital $(a_1,b_1)$ of $f_1$. Hence, there exists $m_1\in \mathbb{Z}$ such that $f_1^{m_1}(I)\prec I_1$. We let $J_1=f_1^{m_1}(I)$. Note that $J_1\prec I_1\prec I$ and that $J_1$ is contained in $(a_1,b)\subseteq (a_2,b_2)$. 
		For $i=2$, since $J_1$ and $I$ are contained in the orbital $(a_2,b_2)$ of $f_2$, there exists $n_2\in\mathbb{Z}$ such that $f_2^{n_2}(I)\prec J_1$. We let $I_2=f_2^{n_2}(I)$ and note that $I_2$ is also contained in the orbital $(a_2,b_2)$ of $f_2$. Hence, there exists $m_2\in \mathbb{Z}$ such that $f_2^{m_2}(I)\prec I_2$. We let $J_2=f_2^{m_2}(I)$. Note that $J_2\prec I_2\prec J_1\prec I_1\prec I$ and that $J_2$ is contained in $(a_2,b)\subseteq (a_3,b_3)$.
		We continue in this manner  %for $i\in\{1,\dots,k\}$ 
		 until we get integers $n_1,m_1,\dots,n_k,m_k\in\BZ$  and closed intervals $I_1,J_1,\dots,I_k,J_k$ such that the following conditions hold. 
		\begin{enumerate}
			\item[(1)] For each $i\in\{1,\dots,k\}$ we have $f_i^{n_i}(I)=I_i$ and $f_i^{m_i}(I)=J_i$.
			\item[(2)] $J_k\prec I_k \prec J_{k-1}\prec I_{k-1}\prec\cdots\prec J_1\prec I_1\prec I$.
		\end{enumerate} 
		Now, consider the set of functions $L=\{f_1^{n_1},f_1^{m_1}, \dots f_k^{n_k},f_k^{m_k}\}$. Since every function in $F$ is piecewise linear with finitely many pieces, 
		%By the definition of Thompson's group $F$, 
		there  exists a (long enough) finite binary word $w$ such that the dyadic interval $[w]$ is contained in $I$ and such that all the functions in $L$ are linear on $[w]$ (and as such, map it onto dyadic intervals). Hence, for each $i\in\{1,\dots,k\}$, there exist finite binary words $u_i$ and $v_i$ such that
		$f_i^{m_i}$ has the pair of branches $w\to u_i$ and $f_i^{n_i}$ has the pair of branches $w\to v_i$. 
		Note that $[u_i]\subseteq f_i^{m_i}(I)= J_i$ and $[v_i]\subseteq f_i^{n_i}(I)= I_i$. Hence,  
		%	$[u_i]=f_i^{m_i}([w])$ and $[v_i]=f_i^{n_i}([w])$ and
		% Note that $[u_i]\subseteq J_i$ and $[v_i]\subseteq I_i$ 
		$$[u_k]\prec[v_k]\prec[u_{k-1}]\prec[v_{k-1}]\prec\cdots\prec[u_1]\prec[v_1]\prec[w]$$
		and in particular, the words $u_1,v_1,\dots,u_k,v_k,w$ are incomparable.
		In addition, since $f_i^{m_i}$ has the pair of branches $w\to u_i$ and $f_i^{n_i}$ has the pair of branches $w\to v_i$ we have 
		$w\sim_{\la f_i\ra} u_i$ and $w\sim_{\la f_i\ra} v_i$. Hence, 
		% of $u_i$ and $v_i$ implies that for  $i\in\{1,\dots,k\}$ we have 
		  $u_i\sim_{\la f_i \ra} v_i \sim_{\la f_i \ra} w$, as required.

	\end{proof}
	
	The proof of the following lemma is similar  to the proof of Lemma \ref{lem:obvious} and is left as an exercise to the reader. 
	
	\begin{lemma}\label{lem:incomparable}
		Let $S=\{f_1,\dots,f_k\}$ be a subset of non-trivial elements of $F$. 
	Then there exist incomparable finite binary words $u_1,v_1,\dots, u_k,v_k\in \mathcal B'$ such that
		for each $i\in\{1,\dots,k\}$ we have that   $u_i\sim_{\la f_i \ra} v_i$.
	\end{lemma}
	
\section{Proof of the main result}

To prove the main result we consider $3$  cases in the following $3$ propositions. 
	
\begin{proposition}\label{cd<0}
	Let $S$ be a finite subset of $F$ and let $g\in F$ be an element which does not fix any number in $(0,1)$. Assume that $\pi(g)$ is a common co-generator of the elements of $\pi(S)$ in $\Ztwo$.
 Then there exists an element $\sigma\in F$ such that $g^\sigma$ is a common co-generator of the elements of $S$ in $F$. 
\end{proposition}

	\begin{proof}
		Since $g$ does not fix any number in $(0,1)$, the interval $(0,1)$ is an orbital of $g$. We can assume that $(0,1)$ is a push-up orbital of $g$, by replacing $g$ by its inverse if necessary (note that if for some $\sigma\in F$ the element $(g^{-1})^\sigma$ is a 
		  common co-generator of the elements of $S$ in $F$ then its inverse $g^\sigma$ is also a common co-generator of the elements of $S$ in $F$).

		Let us denote the elements of $S$ by $f_1,\dots,f_k$. %$S=\{f_1,\dots,f_k\}$.
		By Lemma \ref{lem:incomparable}, there exist incomparable finite binary words $u_1,v_1,\dots,u_k,v_k\in \mathcal B'$
	 such that $u_i\sim_{\la f_i\ra} v_i$ for each $i\in\{1,\dots,k\}$.
		
		Let $T$ be a finite binary tree which has the branches $u_i,v_i0,v_i1$ for all  $i\in\{1,\dots,k\}$ (note that such a tree exists).
		Let $w_1,\dots,w_n$ be the branches of $T$ and note that the intervals $[w_2],\dots,[w_{n-1}]$ are consecutive dyadic sub-intervals of $(0,1)$. Hence, by Lemma \ref{choice}, there exists a function $h\in F$ with the pairs of branches $w_i\to w_{i+1}$ for $i\in\{2,\dots,n-2\}$.
		
		Note that for every $x\in [w_2]\cup\dots\cup [w_{n-2}]=[.w_2,.w_{n-1}]$ %\bigcup_{i=2}^{n-1} [w_i]$
		 we have $h(x)>x$. Hence, the interval $[.w_2,.w_{n-1}]$ is contained in some push-up orbital $(a,b)$ of $h$.
		Now, since  $(0,1)$ is a push-up orbital of $g$, by Lemma \ref{lem:con}, there exists an element $\sigma\in F$ such that $g^\sigma$ coincides with $h$ on the interval $[.w_2,.w_{n-1}]$. %Let $g_1=g^\sigma$ and note that
		Hence, $g^\sigma$
		 has the pairs of branches $w_i\to w_{i+1}$ for all $i\in\{2,\dots,n-2\}$. Note that since $(0,1)$ is an orbital of $g$, it is also an orbital of $g^\sigma$. 
		  Note also that for all $i\in\{2,\dots,n-1\}$ the interval $[w_i)\subseteq (0,1)$ is a fundamental domain of $g^\sigma$. Indeed, for $2\leq i \leq n-1$, we have  $g^{\sigma}(.w_i)=.w_{i+1}$. Hence,  $[w_i)=[.w_i,.w_{i+1})=[.w_i,g^\sigma(.w_i))$, so  $[w_i)$ is indeed a fundamental domain of $g^\sigma$. 
		  
		  We claim that $g^\sigma$ is a common co-generator of the elements of $S$ in $F$. 
		
		Let $j\in\{1,\dots,k\}$.	 We will prove that $\{f_j,g^\sigma\}$ is a generating set of $F$. Recall that $u_j,v_j0,v_j1$ are inner branches of $T$  and that $u_j\sim_{\la f_j\ra} v_j$. For simplicity, we will denote $f_j$ by $f$ and the binary words $u_j$ and $v_j$ by $u$ and by $v$, respectively, for the remainder of the proof. 
		Let $H$ be the subgroup of $F$ generated by $f$ and by $g^\sigma$. It suffices to prove that $H=F$.
		
		 First, note that by assumption, $\pi(g^\sigma)=\pi(g)$ is a co-generator of $\pi(f)$ in $\mathbb{Z}^2$. Hence, $\pi(H)=\mathbb{Z}^2$, so $H[F,F]=F$. Hence, by Theorem \ref{gen}, to prove that $H=F$ it suffices to prove that $\Cl(H)$ contains the derived subgroup of $F$. For that, we will make use of Lemma \ref{lem:all inner words identified}. 
		
		Let us denote the equivalence relation $\sim_{\Cl(H)}$ by $\sim$. Since $g^\sigma$ has the pairs of branches $w_i\to w_{i+1}$ for $i\in\{2,\dots,n-2\}$ , we have $w_2\sim w_3 \sim \dots\sim w_{n-1}$. 
		In particular, since $u,v0,v1$ are inner branches of $T$ (and as such belong to the set $\{w_2,\dots,w_{n-1}\}$), we have $u\sim v0\sim v1$. Since $u\sim v$ (indeed, we have $u\sim_{\la f\ra} v$ and $f\in H$), we get that $u\sim u0\sim u1$. Finally, recall that $(0,1)$ is an orbital of $g^\sigma$ and that each of the intervals $[w_i)$ for $i\in\{2,\dots,n-1\}$ is a fundamental domain of $g^\sigma$. In particular, since $u\in \{w_2,\dots,w_{n-1}\}$, the interval $[u)\subseteq (0,1)$ is a fundamental domain of $g^\sigma$. Since $u\sim u0\sim u1$, by Lemma \ref{lem:fun_dom}, every finite binary word $q$ such that $[q]$ is contained in $(0,1)$ is $\sim$-equivalent to $u$. In other words, all finite binary words in $\mathcal B'$ are $\sim$-equivalent. Hence, by Lemma \ref{lem:all inner words identified}, we have $[F,F]\leq \Cl(H)$, as necessary. 
\end{proof}

\begin{proposition}\label{cd neq 0}
	Let $S$ be a finite subset of $F$ and assume that there is a number in $(0,1)$ that is not fixed by any of the elements in $S$. Let $g\in F$ be such that $\pi(g)$ is a common co-generator of the elements of $S$ in $F$ and such that $g$ fixes some number in $(0,1)$. Assume in addition that $\pi(g)=(c,d)$ such that $cd\neq 0$.
	 Then there exists an element $\sigma\in F$ such that $g^\sigma$ is a common co-generator of the elements of $S$ in $F$. 
\end{proposition}

\begin{proof}
	Since $\pi(g)=(c,d)$ and $c,d\neq 0$, the element $g$ does not fix a right neighborhood of $0$ and does not fix a left neighborhood of $1$. Since $g$ fixes some number in $(0,1)$, the element $g$ has at least two orbitals: $(0,\alpha)$ and $(\beta,1)$, for some $\alpha\leq \beta$. We can assume that $(0,\alpha)$ is a push-up orbital of $g$, by replacing $g$ by its inverse if necessary. We will also assume that $(\beta,1)$ is a push-down orbital of $g$, the proof in the other case being similar.

	Let us denote the elements of $S$ by $f_1,\dots,f_k$. By Lemma \ref{lem:obvious}, since there is a number in $(0,1)$ that is not fixed by any of the elements in $S$, 
	there exist incomparable finite binary words $u_1,v_1,\dots,u_k,v_k,w\in \mathcal B'$ such that
		$[u_1], [v_1],\cdots, [u_k],[v_k]\prec [w]$
	and such that for each $i\in\{1,\dots,k\}$ we have $u_i\sim_{\la f_i\ra} v_i\sim_{\la f_i\ra} w$. 
	
	Let $T$ be a finite binary tree which has the branch $w00$ %,w01,w10,w11$
	and the branches $u_i,v_i0,v_i1$ for all  $i\in\{1,\dots,k\}$ (note that such a tree exists).
	Let $w_1,\dots,w_n$ be the branches of $T$ and note that the intervals $[w_2],\dots,[w_{n-1}]$ are consecutive dyadic sub-intervals of $(0,1)$. 
	Let $m$ be such that $w_m\equiv w00$ (i.e., such that the $m^{th}$ branch of $T$ is $w00$). 
	By Lemma \ref{choice}, there exists a function $h_1\in F$ with the pairs of branches $w_i\to w_{i+1}$ for all $i\in\{2,\dots,m-1\}$. Let $I=[w_2]\cup\dots\cup [w_{m-1}]=[.w_2,.w]$ and note that  $h_1(I)=[w_3]\cup\dots\cup[w_m]=[.w_3,.w01]$. Since for every $x\in I$, we have $h_1(x)>x$, the interval $I$ % $I\cup h_1(I)=[.w_2,.w01]$ 
	is contained in some push-up orbital $(a,b)$ of $h_1$.

	Since the orbital $(0,\alpha)$ of $g$ is also a push-up orbital, by Lemma \ref{lem:con}, there exists an element $\sigma_1\in F$ such that $g^{\sigma_1}$ coincides with $h_1$ on the interval $I$ and such that $I$ is contained in   $(\sigma_1(0),\sigma_1(\alpha))=(0,\sigma_1(\alpha))$. 
	 Note that since $(0,\alpha)$ is a push-up orbital of $g$, the interval $(0,\sigma_1(\alpha))$ is a push-up orbital of   $g^{\sigma_1}$. 
	 Since $I\subseteq(0,\sigma_1(\alpha))$, we also have $I\cup g^{\sigma_1}(I)=I\cup h_1(I)\subseteq (0,\sigma_1(\alpha))$.
	 %(which contains the interval $I\cup J$) 
	 
	 Since $(\beta,1)$ is a push-down orbital of $g$, the interval
	% Note also that the interval 
	$(\sigma_1(\beta),1)$ is a push-down orbital of $g^{\sigma_1}$. Since $h_1(I)\subseteq (0,\sigma_1(\alpha))$, we  have $h_1(I)\prec (\sigma_1(\beta),1)$.
		 %Clearly, the orbital $(0,\sigma_1(\alpha))$ contains the interval $I\cup h_1(I)$. 
	 Let $h_2\in F$ be an element
	 which fixes the point $.w01$ and has the pair of branches  
	 $w101\to w100$ (such an element exists).
	 Note that
		 for every $x\in [w101]$ we have $h_2(x)<x$. Hence, the interval $[w101]$ is contained in some push-down orbital $(c,d)$ of $h_2$. Since $h_2$ fixes the point $.w01$, we have $c\geq .w01$. In particular, $h_1(I)\prec (c,d)$ (indeed, the right end-point of $h_1(I)$ is $.w01$).

	   Now, since $(c,d)$ and $(\sigma_1(\beta),1)$  are push-down orbitals of $h_2$ and of $g^{\sigma_1}$, respectively, and since $[w101]\subseteq (c,d)$ and $h_1(I)\prec (c,d),(\sigma_1(\beta),1)$, by Lemma \ref{lem:con}, there exists an element $\sigma_2\in F$ which fixes the interval $h_1(I)$ pointwise  such that  $[w101]\subseteq (\sigma_2(\sigma_1(\beta)),1)$ and such that $(g^{\sigma_1})^{\sigma_2}$ coincides with $h_2$ on the interval $[w101]$.
	 
	 Let $\sigma=\sigma_1\sigma_2$ and note that $g^\sigma$ coincides with $h_1$ on the interval $I$ (indeed, $g^{\sigma_1}$ coincides with $h_1$ on the interval $I$ and $\sigma_2$ is the identity on the interval  %$g^{\sigma_1}(I)=
	 $h_1(I)$). 
	 $g^\sigma$ also coincides with $h_2$ on the interval $[w101]$. %, it has the pair of branches $w101\to w100$. 
	 	 Let $\gamma_1=\sigma(\alpha)$ and $\gamma_2=\sigma(\beta)$. Here is a summary of the relevant properties of $g^\sigma$. 
	 \begin{enumerate}
	 	\item[(1)]  $g^\sigma$ has a push-up orbital $(0,\gamma_1)$ and a push-down orbital $(\gamma_2,1)$.
	 	\item[(2)] The interval $I\cup h_1(I)=[w_2]\cup \dots\cup [w_m]=[.w_2,.w01]\subseteq (0,\gamma_1)$.
	 	\item[(3)] For every  $i\in\{2,\dots,m-1\}$, the function $g^\sigma$ has the pairs of branches $w_i\to w_{i+1}$ (since $g^\sigma$ coincides with $h_1$ on $I$).
	 	\item[(4)]  For each $i\in\{2,\dots,m\}$, the interval $[w_i)\subseteq (0,\gamma_1)$ is a fundamental domain of $g^\sigma$ (indeed, it follows from (3)).
	 	% for each $i\in\{2,\dots,m-1\}$ we have $g^\sigma([w_i])=[w_{i+1}]$. %, for each $i\in\{1,\dots,m\}$, the interval $[w_i)$ is a fundamental domain of $g^\sigma$.
	 	\item[(5)] The interval $g^\sigma$ has the pair of branches $w101\to w100$. 
		\item[(6)]  The interval $[w100]\cup [w101]\subseteq (\gamma_2,1)$. (Indeed, $[w101]$ is contained in the orbital $(\gamma_2,1)$. Hence, $g^\sigma([w101])=[w100]$ is also contained in that orbital.)
	 	\item[(7)]	
	 	 The interval $[w101)\subseteq (\gamma_2,1)$ is a fundamental domain of $g^\sigma$ (since $g^\sigma$ has the pair of branches $w101\to w100$, it maps the right endpoint of $[w101)$ to its left endpoint). 
	 \end{enumerate}
 
 	Note that from (1), (2) and (6) it follows that $.w01<\gamma_1<\gamma_2<.w1$.
	 Hence, there exist (possibly infinite) binary words $r_1,r_2$ such that $\gamma_1=.w01r_1$ and $\gamma_2=.w01r_2$. %(in fact $r$ begins with $01$).
	   
	   We claim that $g^\sigma$ is a co-generator of the elements of $S$ in $F$. 
	 
	 Let $f_j\in F$. We will prove that $\{f_j,g^\sigma\}$ is a generating set of $F$. Recall that $u_j,v_j0,v_j1$ are inner branches of $T$ to the left of the branch $w00$ and that $u_j\sim_{\la f_j\ra} v_j\sim_{\la f_j\ra} w$. For simplicity, we will denote $f_j$ by $f$ and the binary words $u_j$ and $v_j$ by $u$ and by $v$, respectively, for the remainder of the proof. 
	 Let $H$ be the subgroup of $F$ generated by $f$ and by $g^\sigma$. It suffices to prove that $H=F$. First note that by assumption, $\pi(g^\sigma)=\pi(g)$ is a co-generator of $\pi(f)$ in $\mathbb{Z}^2$. Hence, $\pi(H)=\mathbb{Z}^2$, so $H[F,F]=F$. Hence, by Theorem \ref{gen}, to prove that $H=F$, it suffices to prove that $\Cl(H)$ contains the derived subgroup of $F$. For that, we will make use of Lemma \ref{lem:all inner words identified}. 
	 
	 Let us denote the equivalence relation $\sim_{\Cl(H)}$ by $\sim$.  Since $f\in H$ and $u\sim_{\la f\ra} v\sim_{\la f\ra} w$, we have $u\sim v\sim w$. Similarly, since $g^\sigma$ has the pairs of branches $w_i\to w_{i+1}$ for $i\in\{2,\dots,m-1\}$ , we have $w_2\sim w_3 \sim \dots\sim w_{m}$. 
	 In particular, since $u,v0,v1$ are inner branches of $T$ to the left of the branch $w00$ (and as such belong to the set $\{w_2,\dots,w_{m-1}\}$), we have $u\sim v0\sim v1$. Since  $u\sim v$, we get that $u\sim u0\sim u1$. Now, recall that $(0,\gamma_1)$ is an orbital of $g^\sigma$ and that for each $i\in\{2,\dots,m-1\}$ the interval $[w_i)\subseteq (0,\gamma_1)$  
	  is a fundamental domain of $g^\sigma$. In particular, the interval $[u)\subseteq (0,\gamma_1)$ is a fundamental domain of $g^\sigma$. Hence, by Lemma \ref{lem:fun_dom}, since $u\sim u0\sim u1$, every finite binary word $q$ such that $[q]$ is contained in $(0,\gamma_1)$ is $\sim$-equivalent to $u$. %Note that since $w\sim_{\la f\ra} u$ and $f\in H$, we also have $w\sim u$. 
	  
	  Now, since $u\sim u0\sim u1$, every descendant of $u$ is $\sim$-equivalent to $u$. Since, $w\sim u$, every descendant of $w$ is  $\sim$-equivalent to $u$. In particular, $w101\sim w\sim u$. Recall that $[w101)\subseteq (\gamma_2,1)$ is a fundamental domain of $g^\sigma$. Since $w101$ is $\sim$-equivalent to each of its descendants, by Lemma \ref{lem:fun_dom}, for every finite binary word $q$ such that $[q]$ is contained in $(\gamma_2,1)$ we have $q\sim w101\sim u$.

	  Note that if $q\in\mathcal B'$ is a finite binary word such that $[q]$ is not contained in $(0,\gamma_1)$ and not contained in $(\gamma_2,1)$, then $[q]$ contains $\gamma_1$ or $\gamma_2$ or satisfies $[q]\subseteq (\gamma_1,\gamma_2)=(.w01r_1,.w01r_2)$. 
	  If $[q]\subseteq (\gamma_1,\gamma_2)$, then the word $w01$ is a prefix of $q$ and as such, $q\sim w\sim u$. Similarly, if the length of $q$ satisfies $|q|\geq |w|$ and $[q]$ contains $\gamma_1$ or $\gamma_2$, then the word $w$ is a prefix of $q$ and as such $q\sim u$. 
	  
	  The last $3$ paragraphs show that for any $q\in \mathcal B'$ such that $|q|\geq |w|$ we have $q\sim u$. It follows that for every $q\in \mathcal B'$ and every finite binary word $s$ such that $|s|=|w|$ we have $qs\sim u$. Hence, by Lemma \ref{lem:clear}, every finite binary word $q\in \mathcal B'$ satisfies $q\sim u$. 
 Therefore, by Lemma \ref{lem:all inner words identified}, we have $[F,F]\leq \Cl(H)$, as necessary. 
	 \end{proof}

\begin{proposition}\label{prop:cd=0}
	Let $(c,d)\in\{(\pm 1,0), (0,\pm 1)\}$ and let $g\in F$ be an element such that $\pi(g)=(c,d)$.  
	Let $S\subseteq F$ be a finite subset such that $(c,d)$ is a common co-generator  of the elements of $\pi(S)$ in $\Ztwo$.  Then there exists an element $\sigma\in F$ such that $g^\sigma$ is a common co-generator of the elements of $S$ in $F$. 
\end{proposition}

\begin{proof}
	We will prove the proposition in the case where $(c,d)=(1,0)$, the other cases being similar. 
	Let $f_1,\dots,f_k$ be the elements of $S$.
	Since $(1,0)$ is a common co-generator of the elements of $\pi(S)$ in $\Ztwo$, for each $i\in\{1,\dots,k\}$ the second component of $\pi(f_i)$ is $\pm 1$. Hence, for each $i\in \{1,\dots,k\}$ we have $f_i'(1^-)=2^{\pm 1}$. It follows that for each $i\in\{1,\dots,k\}$ there exists $n_i\in \mathbb{Z}$ such that $f_i$ or $f_i^{-1}$ has the pair of branches $1^{n_i}\to 1^{n_i+1}$. % and thus $1^{n_i}\sim_{\la f_i\ra} 1^{n_i+1}$. 
	 Let $n=\max\{n_i\mid i\in\{1,\dots,k\}\}$ and note that for each $i\in\{1,\dots,k\}$ we have $1^n\sim_{\la f_i\ra} 1^{n+1}$. Note also that since $\pi(g)=(1,0)$, we have $g'(0^+)=2$. Hence,  the element $g$ has a push-up orbital of the form $(0,\alpha)$ for some 
	 $\alpha\in (0,1)$.

	By Lemma \ref{lem:incomparable}, there exist incomparable finite binary words $u_1,v_1,\dots,u_k,v_k\in \mathcal B'$
	such that $u_i\sim_{\la f_i\ra} v_i$ for all $i\in\{1,\dots,k\}$. 
		Let $m\in \mathbb{N}$ be a positive integer such that $m\geq  n+1$ and such that $1^m$ is not a prefix of any of the words $u_1,v_1,\dots,u_k,v_k$ and note that $[u_1],[v_1],\dots,[u_k],[v_k]\prec [1^m]$. 
	
	Let $T$ be a finite binary tree which has the branch $1^m$ and the branches $u_i,v_i0,v_i1$ for all  $i\in\{1,\dots,k\}$. %(note that such a tree exists).
	Let $w_1,\dots,w_n$ be the branches of $T$ and note that the intervals $[w_2],\dots,[w_{n-1}]$ are consecutive dyadic sub-intervals of $(0,1)$. Hence, by Lemma \ref{choice} there exists a function $h\in F$ with the pairs of branches $w_i\to w_{i+1}$ for $i\in\{2,\dots,n-2\}$.
	
	Let $I=[w_2]\cup\dots\cup[w_{n-2}]$ and note that 
	for every $x\in I$ we have $h(x)>x$. Hence, the interval $I\cup h(I)=[w_2]\cup\dots\cup [w_{n-1}]=[.w_2,.w_{n}]=[.w_2,.1^m]$ is contained inside some push-up orbital $(a,b)$ of $h$.
	Recall that $(0,\alpha)$ is a  push-up orbital of the form $(0,\alpha)$ of $g$.  Hence, by Lemma \ref{lem:con}, there exists an element $\sigma\in F$ such that $g^\sigma$ coincides with $h$ on the interval $I$ and such that
	 $I\subseteq (0,\sigma(\alpha))$.		 
		 Since $(0,\sigma(\alpha))$ is a push-up orbital of $g^\sigma$ which contains $I$, it also contains $g^\sigma(I)=h(I)$. Hence, $\sigma(\alpha)>.1^m$ (since $.1^m$ is the right end-point of $h(I)$).
	 Note also that since for each $i\in\{2,\dots,n-2\}$, we have $g^\sigma([w_i])=[w_{i+1}]$, the interval $[w_i)$ is a fundamental domain of $g^\sigma$ for each $i\in\{2,\dots,n-1\}$.

	We claim that $g^\sigma$ is a co-generator of the elements of $S$ in $F$. 
	
	Let $f_j\in S$. We will prove that $\{f_j,g^\sigma\}$ is a generating set of $F$. Recall that $u_j,v_j0,v_j1$ are inner branches of $T$ and that $u_j\sim_{\la f_j\ra} v_j$. For simplicity, we will denote $f_j$ by $f$ and the binary words $u_j$ and $v_j$ by $u$ and by $v$, respectively, for the remainder of the proof. 
	Let $H$ be the subgroup of $F$ generated by $f$ and by $g^\sigma$. It suffices to prove that $H=F$. First note that by assumption, $\pi(g^\sigma)=(1,0)$ is a co-generator of $\pi(f)$ in $\mathbb{Z}^2$. Hence, $\pi(H)=\mathbb{Z}^2$, so $H[F,F]=F$. Hence, by Theorem \ref{gen}, it suffices to prove that $\Cl(H)$ contains the derived subgroup of $F$. For that, we will make use of Lemma \ref{lem:all inner words identified}.

	Let us denote the equivalence relation $\sim_{\Cl(H)}$ by $\sim$. Since $g^\sigma$ has the pairs of branches $w_i\to w_{i+1}$ for $i\in\{2,\dots,n-2\}$ , we have $w_2\sim w_3 \sim \dots\sim w_{n-1}$. 
	In particular, since $u,v0,v1$ are inner branches of $T$  (and as such belong to the set $\{w_2,\dots,w_{n-1}\}$), we have $u\sim v0\sim v1$. Since $u\sim v$, we get that $u\sim u0\sim u1$. Now, recall that $(0,\sigma(\alpha))$ is an orbital of $g^\sigma$ and that each of the intervals $[w_i)\subseteq (0,\sigma(\alpha))$ for $i\in\{2,\dots,n-1\}$ is a fundamental domain of $g^\sigma$. In particular, the interval $[u)\subseteq (0,\sigma(\alpha))$ is a fundamental domain of $g^\sigma$. Hence, by Lemma \ref{lem:con}, every finite binary word $q$ such that $[q]$ is contained in $(0,\sigma(\alpha))$ is $\sim$-equivalent to $u$.
	Since $\sigma(\alpha)>.1^m$,it follows that for every finite binary word $q$ such that $[q]\subseteq (0,.1^m]$ we have $q\sim u$. 
	
	Now, let $q\in  B'$ be such that $|q|\geq m$ and such that 
	 $[q]\not\subseteq (0,.1^m]$. 
	 Since $[q]\not\subseteq (0,.1^m]$ and $|q|\geq m$, %we have $[q]\subseteq (.1^m,1)$ or $.1^m\in [q]$. Note that in both cases
	 the word $1^m$ must be a prefix of $q$. Hence, there exists some $\ell\geq m$ and a (possibly empty) finite binary word $r$ such that $q\equiv 1^\ell0r$. Note that $\ell\geq m\geq n+1$ and that $1^{n+1}\sim 1^n$. Using the fact that $1^{n+1}\sim 1^n$ repeatedly, we get that $1^\ell\sim 1^n$. Hence, $q\sim 1^n0r$.  Since $n<m$, the interval $[1^n0r]$ is contained in $(0,.1^m]$ and as such, $1^n0r\sim u$, so $q\sim u$. 
	 
	 Now, the last two paragraphs show that every word $q\in \mathcal B'$ such that $|q|\geq m$ satisfies $q\sim u$. As in the proof of Proposition \ref{cd neq 0}, that implies that all finite binary words in $\mathcal B'$ are $\sim $-equivalent to $u$. Hence, by Lemma \ref{lem:all inner words identified}, we have $[F,F]\leq \Cl(H)$, as necessary. 
\end{proof}

Now, we are ready to prove Theorem \ref{thm:most general}. Recall the theorem.  

\begin{theorem}
	Let $S$ be a finite subset of $F$ and let $g\in F$. Assume that $\pi(g)$ is a common co-generator of the elements of $\pi(S)$ in $\Ztwo$. Then there is an element $\sigma\in F$ such that $g^\sigma$ is a common co-generator of the elements of $S$ in $F$ if and only if at least one of the following conditions holds.
	\begin{enumerate}
		\item The element $g$ does not fix any number $\alpha\in (0,1)$. 
		\item There is a number $\beta\in (0,1)$ that is not fixed by any of the elements in $S$. 
	\end{enumerate}
\end{theorem}

\begin{proof}
	One direction follows from Lemma \ref{lem:only obstruction}. Let us  prove the other direction. Assume that  Condition (1) or Condition (2) holds.
	If Condition (1) holds, then by Proposition \ref{cd<0}, there exists $\sigma\in F$ such that $g^\sigma$ is a common co-generator of the elements of $S$ in $F$. Hence, we can assume that Condition (1) does not hold, i.e., that the element $g$ fixes some number in $(0,1)$. In that case, Condition (2) must hold. Let $c,d\in\mathbb{Z}$ be such that $\pi(g)=(c,d)$. If $cd\neq 0$ then all the assumptions of Proposition \ref{cd neq 0} hold. Hence, there exists an element $\sigma\in F$ such that $g^\sigma$ is a common co-generator of the elements of $S$ in $F$. If $cd=0$, then $(c,d)\in\{(0,\pm 1),(\pm 1,0)\}$ (since $(c,d)$ is part of a generating pair of $\Ztwo$). Then the result follows from Proposition \ref{prop:cd=0}.
\end{proof}

\begin{minipage}{3 in}
	Gili Golan\\
	Department of Mathematics,\\
	Ben Gurion University of the Negev,\\ %Be'er Sheva, Israel\\
	golangi@bgu.ac.il
\end{minipage}
\end{document}